\documentclass[runningheads]{llncs}
\usepackage{graphicx}
\usepackage[symbol]{footmisc}

\usepackage{float}
\usepackage{mathtools}
\usepackage{amssymb, amsmath, latexsym}
\usepackage{nicefrac}
\usepackage{hhline}
\usepackage{multirow}
\usepackage{pifont}
\usepackage[colorinlistoftodos,bordercolor=orange,backgroundcolor=orange!20,linecolor=orange,textsize=scriptsize]{todonotes}
\usepackage{algorithm}\usepackage{algpseudocode}

\def \R {\mathbb R}
\def\eqdef{\overset{\text{def}}{=}}

\newcommand{\EndProof}{\begin{flushright}$\square$\end{flushright}}

\newtheorem{assumption}{Assumption}
  
\newcommand{\EE}{\mathbf{E}}

\def\R{\mathbb{R}}
\newcommand{\E}{{\mathbb E}}

\def\R{\mathbb R}
\def\E{\mathbb E}
\def\EE{\mathbb E}

\def\la{\langle}
\def\ra{\rangle}

\begin{document} 
\title{One-Point Gradient-Free Methods for Smooth and Non-Smooth Saddle-Point Problems\thanks{
The research of A. Beznosikov and A. Gasnikov in Algorithm 1, Theorems 1-3 was supported by Russian Science Foundation (project No. 21-71-30005). The research of V. Novitskii in Algorithms 2, Theorems 4-7 was partially supported by Andrei Raigorodskii scholarship.
}}
\titlerunning{One-Point Gradient-Free Methods for Saddle-Point Problems}
%
\author{Aleksandr Beznosikov\inst{1,2}\and
 Vasilii Novitskii\inst{1} \and
Alexander Gasnikov\inst{1,2}}
\authorrunning{A. Beznosikov, V. Novitskii, A. Gasnikov}
%
\institute{Moscow Institute of Physics and Technology, Dolgoprudny, Russia \and
Higher School of Economics, Russia 
}
\maketitle              
\begin{abstract}

In this paper, we analyze gradient-free methods with one-point feedback for stochastic saddle point problems $\min_{x}\max_{y}  \varphi(x, y)$. For non-smooth and smooth cases, we present an analysis in a general geometric setup with arbitrary Bregman divergence. For problems with higher order smoothness, the analysis is carried out only in the Euclidean case. The estimates we have obtained repeat the best currently known estimates of gradient-free methods with one-point feedback for problems of imagining a convex or strongly convex function. The paper uses three main approaches to recovering the gradient through finite differences: standard with a random direction, as well as its modifications with kernels and residual feedback. We also provide experiments to compare these approaches for the matrix game.

\end{abstract}
\section{Introduction}

This paper is devoted to solving the saddle-point problem:
\begin{equation}
    \label{problem}
    \min_{x\in \mathcal{X}}\max_{y\in \mathcal{Y}}  \varphi(x, y).  
\end{equation}
It has many practical applications. These are the already well-known and classic matrix game and Nash equilibrium, as well as modern machine learning problems: Generative Adversarial Networks (GANs) \cite{goodfellow2016nips} and Reinforcement Learning (RL) \cite{pmlr-v119-jin20f}. We assume that only zeroth-order information about the function is available, i.e. only its values, not a gradient, hessian, etc. This concept is called a Black-Box and arises in optimization \cite{Nesterov}, adversarial training \cite{Chen_2017}, RL \cite{fazel2018global}. To make the problem statement more complex, but close to practice, it is natural to assume that we have access inexact values of function $\varphi(x, y, \xi)$, for example, with some random noise $\xi$. But even with the help of such an oracle, it is possible to recover some estimate of the gradient of a function in terms of finite differences.

Let us highlight two main approaches to such gradient estimates. The first approach is more well researched in the literature and is called a two-point feedback:
$$\frac{n}{2\tau} (\varphi(x+  \tau \mathbf{e}_x, y+  \tau \mathbf{e}_y, \xi)  - \varphi(x - \tau \mathbf{e}_x, y-  \tau \mathbf{e}_y, \xi)) \left(
    \begin{array}{c}
    \mathbf{e}_x\\
    -\mathbf{e}_y \\
    \end{array}
    \right).$$
An important feature of this approach is that it is assumed that we were able to obtain the values of the function in points $(x+  \tau \mathbf{e}_x, y+  \tau \mathbf{e}_y)$ and $(x - \tau \mathbf{e}_x, y-  \tau \mathbf{e}_y)$ with the same realization of the noise $\xi$. From the point of view of theoretical analysis, such an assumption is strong and gives good guarantees of convergence \cite{duchi2013optimal,Shamir15,Nesterov}. But from a practical point of view, this is a very idealistic assumption. Therefore, it is proposed to consider the concept of one-point feedback (which this paper is about):
$$\frac{n}{2\tau} (\varphi(x+  \tau \mathbf{e}_x, y+  \tau \mathbf{e}_y, \xi^+)  - \varphi(x - \tau \mathbf{e}_x, y-  \tau \mathbf{e}_y, \xi^-)) \left(
    \begin{array}{c}
    \mathbf{e}_x\\
    -\mathbf{e}_y \\
    \end{array}
    \right).$$
In general $\xi^+ \neq \xi^-$. As far as we know, the use of methods with one-point approximation for saddle-point problems has not been studied at all in the literature. This is the main goal of our work.

\subsection{Related works} 

Since the use of one-point feedback for saddle-point problems is new in the literature, we present related papers in two categories: two-point gradient-free methods for saddle-point problems, and one-point methods for minimization problems. Partially the results of these works are transferred to Table \ref{summary_1}.

\textbf{Two-point for saddle-point problems.} Here, we first highlight work for non-smooth saddle-point problems \cite{aleks2020gradientfree}, as well as work for smooth ones \cite{sadiev2020zeroth}. Note that in these papers an optimal estimate was obtained in the non-smooth case, and in the smooth case only for a special class of "firmly smooth" saddle-point problems. Also note the work devoted to coordinated methods for matrix games \cite{carmon2020coordinate}, which is also close to our topic.

\textbf{One-point for minimization problems.} 
First of all, we present works that analyze functions with higher order smoothness:
\cite{bach2016highly,akhavan2020exploiting,novitskii2021improved}. These works are united by  the technique of special random kernels, which allow you to use the smoothness of higher orders. Note that there is an error in work \cite{bach2016highly}, therefore Table 1 shows the corrected result (according to the note from \cite{akhavan2020exploiting}). The special case of higher order smoothness is also interesting -- the ordinary smoothness, it is also analyzed in \cite{bach2016highly,akhavan2020exploiting,novitskii2021improved}, in addition we note the papers \cite{gasnikov2017stochastic,zhang2020improving}. A nonsmooth analysis is presented in \cite{gasnikov2017stochastic,zhang2020improving}. Note that in paper \cite{gasnikov2017stochastic}, not only the Euclidean setup is analyzed, but also the general case with an arbitrary Bregman divergence, which gives additional advantages in the estimates of the convergence (see Table \ref{summary_1}).

\subsection{Our contribution}

In the nonsmooth case, we consider convex-concave and strongly-convex-strongly-concave problems with bounded $\nabla_x \varphi(x,y)$, $\nabla_y \varphi(x,y)$ on the optimization set. Our algorithm is modofocation of Mirror Descent with arbitrary Bregman divergence. The estimates we obtained coincide with the estimates for convex optimization with one-pointed feedback \cite{gasnikov2017stochastic,zhang2020improving}. Using the correct geometry helps to reduce the contribution of the problem dimension to the final convergence estimate. In particular, in the entropy setting, convergence depends on the dimension of the problem linearly (see Table \ref{summary_1} for more details in convex-concave case and Table \ref{summary_2} -- in strongly-convex-strongly-concave). 

In the smooth case we obtained the estimates of the convergence rate with arbitrary Bregman divergence for convex-concave case and in Euclidean setup for strongly-convex-strongly-concave case. These estimates also coincide with the estimates for convex optimization with one-point feedback \cite{gasnikov2017stochastic}. 

To the best of our knowledge this is the first time when exploiting higher-order smoothness helps to improve performance in saddle-point problems in both  strongly-convex-strongly-concave and convex-concave cases. The results also coincide with the estimates for minimization \cite{novitskii2021improved,akhavan2020exploiting}. 

In Tables \ref{summary_1} and \ref{summary_2} one can find a comparison of the oracle complexity of known results with zeroth-order methods for saddle-point problems in related works. Factor $q$ depends on geometric setup of our problem and gives a benefit when we work in the Hölder, but non-Euclidean case (use non-Euclidean prox), i.e. $\| \cdot\| = \|\cdot\|_p$ and $p \in [1;2]$, then $\| \cdot \|_* = \| \cdot\|_q$, where $\nicefrac{1}{p} +\nicefrac{1}{q} = 1$. Then $q$ takes values from $2$ to $\infty$, in particular, in the Euclidean case $q=2$, but when the optimization set is a simplex, $q = \infty$. 
In higher-order smooth case we consider functions satisfying so called generalized Hölder condition with parameter $\beta > 2$ (see inequality \eqref{Hölder-condition} below). Note that it is prefer to use higher-order smooth methods rather than smooth methods only if $\beta > 3$.
\renewcommand{\arraystretch}{1.5}
\renewcommand{\tabcolsep}{5pt} 
\begin{table}[h!]
\begin{center}
\begin{tabular}{ccccc}
\textbf{Case}  &  \textbf{Oracle} & \textbf{Prob.} & \textbf{Complexity}  & \textbf{Reference}  \\ \hline
\multirow{3}{*}{non-smooth} & two-point & SP  & $\mathcal{O} \left( n^{\frac{2}{q}} \cdot \varepsilon^{-2}\right)$  & \cite{aleks2020gradientfree}\\ \cline{2-5} 
                  & \multirow{2}{*}{one-point} & Min  & $\mathcal{O} \left( n^{1 + \frac{2}{q}} \cdot\varepsilon^{-4}\right)$  & \cite{gasnikov2017stochastic}  \\ \cline{3-5} 
                  &          & SP & $\mathcal{O} \left( n^{1 + \frac{2}{q}} \cdot \varepsilon^{-4}\right)$  & this paper  \\ \hline
\multirow{3}{*}{smooth} & two-point            & SP  & $\mathcal{O} \left( [n^{\frac{2}{q}} \text{ or } n] \cdot \varepsilon^{-2}\right)$  & \cite{sadiev2020zeroth}  \\ \cline{2-5} 
                  & \multirow{2}{*}{ one-point} & Min& $\tilde{\mathcal{O}} \left(n^{2} \cdot \varepsilon^{-3} \right)$ & \cite{gasnikov2017stochastic} \\ \cline{3-5} 
                  &          & SP  & $\tilde{\mathcal{O}} \left(n^{2} \cdot \varepsilon^{-3} \right)$ & this paper \\ \hline
                  
\multirow{2}{*}{higher order smooth} 
                  & \multirow{2}{*}{ one-point} & Min  & $\tilde{\mathcal{O}} \left(n^{2+\frac{2}{\beta-1}} \cdot \varepsilon^{-2-\frac{2}{\beta-1}} \right)$ & \cite{novitskii2021improved,akhavan2020exploiting} \\ \cline{3-5} 
                  &          & SP  & $\tilde{\mathcal{O}} \left(n^{2+\frac{2}{\beta-1}} \cdot \varepsilon^{-2-\frac{2}{\beta-1}} \right)$ & this paper \\ \hline
\end{tabular}
\end{center}
\caption{Comparison of oracle complexity of one-point/two-point 0th-order methods for non-smooth/smooth \textbf{convex} minimization (Min) and \textbf{convex-concave} saddle-point (SP) problems under different assumptions. $\varepsilon$ means the accuracy of the solution, $n$ -- dimension of the problem, $q = 2$ for the Euclidean case and $q = \infty$ for setup of $\|\cdot\|_1$-norm.}
\label{summary_1}
\end{table}
\begin{table}[h!]
\begin{center}
\begin{tabular}{ccccc}
\textbf{Case}  &  \textbf{Oracle} & \textbf{Prob.} & \textbf{Complexity}  & \textbf{Reference}  \\ \hline
\multirow{2}{*}{non-smooth} 
                  & \multirow{2}{*}{one-point} & Min  & $\mathcal{\tilde O} \left( n^{2} \cdot\varepsilon^{-3}\right)$  & \cite{gasnikov2017stochastic}  \\ \cline{3-5} 
                  &          & SP & $\mathcal{\tilde O} \left( n^{2} \cdot \varepsilon^{-3}\right)$  & this paper  \\ \hline
\multirow{3}{*}{smooth} & two-point            & SP  & $\mathcal{O} \left( n \cdot \varepsilon^{-1} \right)$  & \cite{sadiev2020zeroth}  \\ \cline{2-5} 
                  & \multirow{2}{*}{ one-point} & Min& 
                  $\tilde{\mathcal{O}} \left( n^2 \cdot \varepsilon^{-2} \right)$
                  & \cite{gasnikov2017stochastic} 
                  \\ \cline{3-5} 
                  &          & SP  & 
                  $\tilde{\mathcal{O}} \left( n^2 \cdot \varepsilon^{-2} \right)$
                  & this paper \\ \hline
                  
\multirow{2}{*}{higher order smooth} 
                  & \multirow{2}{*}{ one-point} & Min  & 
                  $\tilde{\mathcal{O}} \left(n^{2+\frac{1}{\beta-1}} \cdot \varepsilon^{-\frac{\beta}{\beta-1}} \right)$
                  & \cite{novitskii2021improved,akhavan2020exploiting} 
                  \\ \cline{3-5} 
                  &          &                  SP  & 
                  $\tilde{\mathcal{O}} \left(n^{2+\frac{1}{\beta-1}} \cdot \varepsilon^{-\frac{\beta}{\beta-1}} \right)$
                  & this paper \\ \hline
\end{tabular}
\end{center}
\caption{Comparison of oracle complexity of one-point/two-point 0th-order methods for non-smooth/smooth \textbf{strongly-convex} minimization (Min) and \textbf{strongly-convex-strongly-concave} saddle-point (SP) problems under different assumptions.}
\label{summary_2}
\end{table}

\section{Preliminaries}

To begin with, we introduce some notation and definitions that we use in the work.

\subsection{Notation}

We use $\la x,y \ra \eqdef \sum_{i=1}^nx_i y_i$ to denote inner product of $x,y\in\R^n$ where $x_i$ is the $i$-th component of $x$ in the standard basis in $\R^n$. Then it induces $\ell_2$-norm in $\R^n$ in the following way $\|x\|_2 \eqdef \sqrt{\la x, x \ra}$. We define $\ell_p$-norms as $\|x\|_p \eqdef \left(\sum_{i=1}^n|x_i|^p\right)^{\nicefrac{1}{p}}$ for $p\in(1,\infty)$ and for $p = \infty$ we use $\|x\|_\infty \eqdef \max_{1\le i\le n}|x_i|$. The dual norm $\|\cdot\|_q$ for the norm $\|\cdot\|_p$ is denoted in the following way: $\|y\|_q \eqdef \max\left\{\la x, y \ra\mid \|x\|_p \le 1\right\}$. Operator $\EE[\cdot]$ is full mathematical expectation and operator $\EE_\xi[\cdot]$ express conditional mathematical expectation. 

\begin{definition}[$\mu$-strong convexity]
Function $f(x)$ is $\mu$-strongly convex w.r.t. $\| \cdot \|$-norm on $\mathcal{X}\subseteq \R^n$ when it is continuously differentiable and there is a constant $\mu > 0$ such that the following inequality holds:
\begin{equation*}
\label{strong convexity}
f(y) \geq f(x) + \langle \nabla f(x), y-x \rangle + \frac{\mu}{2}\|y-x\|^2, \quad \forall\ x, y\in \mathcal{X}.
\end{equation*}
\end{definition}

\begin{definition}[Prox-function]
Function $d(z): \mathcal{Z} \to \mathbb{R}$ is called prox-function if $d(z)$ is $1$-strongly convex w.r.t. $\| \cdot \|$-norm and differentiable on $\mathcal{Z}$. 
\end{definition}
\begin{definition}[Bregman divergence]
Let $d(z): \mathcal{Z} \to \mathbb{R}$ is prox-function. For  any  two  points  $z,w \in \mathcal{Z}$ we define Bregman divergence $V_z(w)$ associated with $d(z)$ as follows: 
\begin{equation*}
    \label{Bregman}
    V_z(w) = d(z) - d(w) - \langle \nabla d(w), z - w \rangle.
\end{equation*}
\end{definition}

We denote the Bregman-diameter $\Omega_{\mathcal{Z}}$ of $\mathcal{Z}$ w.r.t.\ $V_{z_1}(z_2)$ as \\
$\Omega_{\mathcal{Z}} \eqdef \max\{\sqrt{2V_{z_1}(z_2)}\mid z_1, z_2 \in \mathcal{Z}\}$.

\begin{definition}[Prox-operator]
Let $V_z(w)$ Bregman divergence. For all $x \in \mathcal{Z}$ define prox-operator of $\xi$:
\end{definition}
\begin{equation*}
    \text{prox}_x (\xi) = \text{arg}\min_{y \in \mathcal{Z}} \left(V_x(y) + \langle \xi , y \rangle \right).
\end{equation*}

Now we are ready to formally describe the problem statement, as well as the necessary assumptions.

\subsection{Settings and assumptions}

As mentioned earlier, we consider the saddle-point problem \eqref{problem}, where $\varphi(\cdot, y)$ is convex function defined on compact convex set $\mathcal{X} \subset \mathbb{R}^{n_x}$,  $\varphi(x,\cdot)$ is concave function defined on compact convex set $\mathcal{Y} \subset \mathbb{R}^{n_y}$. For convenience, we denote $\mathcal{Z} =  \mathcal{X} \times  \mathcal{Y}$  and then $z\in \mathcal{Z} $ means $z \eqdef (x,y)$, where $x \in \mathcal{X}$,  $y \in \mathcal{Y}$.  When we use $\varphi(z)$, we mean $\varphi(z) = \varphi(x,y)$.

\begin{assumption}[Diameter of $\mathcal{Z}$] 
Let the compact set $\mathcal{Z}$ have diameter $\Omega$.
\end{assumption}

\begin{assumption}[$M$-Lipschitz continuity]
Function $\varphi(z)$ is $M$-Lipschitz continuous in certain neighbourhood of $\mathcal{Z}$ with $M > 0$ w.r.t. norm $\|\cdot\|_2$ when 
\begin{equation*}
| \varphi(z)- \varphi(z')|\leq M\|z-z'\|_2, \quad \forall\ z, z'\in \mathcal{Z}.
\end{equation*}
\end{assumption}
One can prove that for all $z\in \mathcal{Z}$ we have 
\begin{equation}
\label{bound}
\|\tilde \nabla \varphi(z)\|_2 \leq M.
\end{equation}
\begin{assumption}[$\mu$-strong convexity--strong concavity]
Function $\varphi(z)$ is $\mu$-strongly-convex-strongly-concave in $\mathcal{Z}$ with $\mu > 0$ w.r.t. norm $\|\cdot\|_2$ when  $\varphi(\cdot, y)$ is $\mu$-strongly-convex for all $y$ and  $\varphi(x, \cdot)$ is $\mu$-strongly-concave for all $x$ w.r.t. $\| \cdot \|_2$.
\end{assumption}

Hereinafter, by $\tilde \nabla \varphi(z)$ we mean a block vector consisting of two vectors $\nabla_x \varphi(x,y)$ and $-\nabla_y \varphi(x,y)$.
Recall that we do not have access to oracles $\nabla_x {\varphi}(x,y)$ or $\nabla_y {\varphi}(x,y)$. We only can use an inexact stochastic zeroth-order oracle $\tilde{\varphi}(x, y, \xi, \delta)$ at each iteration. Our model corresponds to the case when the oracle gives an inexact noisy function value. We have stochastic unbiased noise, depending on the random variable $\xi$ and biased deterministic noise $\delta$. One can write it the following way:
\begin{eqnarray}
\label{PrSt1}
    \tilde \varphi (x,y, \xi)  = \varphi(x,y) + \xi + \delta(x,y).
\end{eqnarray}
Note that $\delta$ depends on point $(x, y)$, and $\xi$ is generated randomly regardless of this point.

\begin{assumption}[Noise restrictions]\label{noise_restrictions} Stochastic noise $\xi$ is unbiased with bounded variance, $\delta$ is bounded, i.e. there exists $\Delta, \sigma >0$ such that
\begin{eqnarray}
\label{eq:PrSt2}
    \EE\xi =0, ~~~~~\EE\left[\xi^2\right] \leq \sigma^2, ~~~~~ |\delta|\leq \Delta.
\end{eqnarray}
\end{assumption}

\section{Theoretical results}

Since we do not have access to $\nabla_x {\varphi}(x,y)$ or $\nabla_y {\varphi}(x,y)$, it is proposed to replace them with finite differences. We present two variants: using a random euclidean direction \cite{Shamir15,gasnikov2017stochastic} in non-smooth case and a kernel approximation \cite{akhavan2020exploiting,novitskii2021improved} in smooth. These two concepts will be discussed in more detail later in the respective sections. As mentioned earlier, we work with one-point feedback.
We use Mirror Descent as the basic algorithm, but with approximations instead of gradient.

\subsection{Non-smooth case} \label{31}

\textbf{Random euclidean direction.} For $\mathbf{e} \in \mathcal{RS}^n_2(1)$ (a random vector uniformly distributed on the Euclidean unit sphere) and some constant $\tau$ let $\tilde \varphi(z + \tau \mathbf{e}, \xi) \eqdef \tilde \varphi(x + \tau \mathbf{e}_x, y+ \tau \mathbf{e}_y, \xi)$, where $\mathbf{e}_x$ is the first part of $\mathbf{e}$ size of dimension $n_x$, and $\mathbf{e}_y$ is the second part of dimension $n_y$. Then define estimation of the gradient through the difference of functions:
\begin{equation}
    \label{eq:PrSt3}
    g(z, \mathbf{e}, \tau,  \xi^{\pm}) = \frac{n \left(\tilde \varphi(z + \tau \mathbf{e},  \xi^+) -  \tilde \varphi(z - \tau \mathbf{e},  \xi^-)\right)}{2\tau}\left(
    \begin{array}{c}
    \mathbf{e}_x\\
    -\mathbf{e}_y \\
    \end{array}
    \right), 
\end{equation}
\begin{minipage}{0.53\linewidth}
     \begin{algorithm}[H]
	\caption{{\tt zoopMD}}
	\label{alg1}
\begin{algorithmic}
\State 
\noindent {\bf Input:} $z_0$, $N$, $\gamma$, $\tau$.
\For {$k=0,1, 2, \ldots, N$ }
    \State $z_{k+1} = \text{prox}_{z_k}(\gamma_k\cdot g(z_k, \mathbf{e}_k, \tau,  \xi^{\pm}_k)$.
\EndFor
\State 
\noindent {\bf Output:} $\bar z_{N}$.
\end{algorithmic}
\end{algorithm}
\end{minipage}
\begin{minipage}{0.05\textwidth}
\end{minipage}
\begin{minipage}{0.46\textwidth}
where $n = n_x +n_y$. It is important that $\xi^+$ and $\xi^-$ are different variables -- this corresponds to the one-point concept. 
Next, we present Algorithm \ref{alg1} -- a modification of Mirror Descent with \eqref{eq:PrSt3}. Note that any Bregman divergence can be used in the prox operator. This allows us to take into ac-
\end{minipage}\\
count the geometric setup of the problem. $\mathbf{e}_k$ and $\xi^{\pm}_k$ are generated independently of the previous iterations and of each other. Here $\bar z_{N} = \frac{1}{N+1} \sum_{i=0}^{N} z_{i}$. Below we give technical facts about \eqref{eq:PrSt3}. Note that we do not provide proofs in the main part of the paper, they are all in the Appendix.

\begin{lemma} [see Lemma 2 from \cite{Beznosikov} or Lemma 1 from \cite{aleks2020gradientfree}]
\label{beznos}
    For $g(z, \mathbf{e}, \tau,  \xi^{\pm})$ defined in \eqref{eq:PrSt3} under Assumptions 2 and 4 the following inequality holds:
    \begin{eqnarray}
    \label{boundlem0}
        \mathbb{E}\left[ \|g(z, \mathbf{e}, \tau,  \xi^{\pm})\|^2_q\right] \leq 3a_q^2 \left(3nM^2+ \frac{n^2(\sigma^2 + \Delta^2)}{\tau^2} \right),
    \end{eqnarray}
    where $a^2_q$ is determined by $\EE[\|e\|_q^2] \leq \sqrt{\EE[\|e\|_q^4]} \le a^2_q$ and the following statement is true
    \begin{eqnarray}
    a_q^2 = \min\{2q - 1, 32 \log n - 8\} n^{\frac{2}{q} - 1}, \quad \forall n \geq 3.\label{eq:condition_on_u}
\end{eqnarray}
\end{lemma}

Next we define an important object for further theoretical discussion -- a smoothed version of the function $\varphi$ (see \cite{Nesterov,Shamir15}). 

\begin{definition}
Function $\hat{\varphi}(z)$ defines on set $\mathcal{Z}$ satisfies:
\begin{equation}
    \label{eq:smoothphi}
    \hat{\varphi}(z) = \mathbb{E}_{\mathbf{e}}\left[\varphi(z+ \tau \mathbf{e}) \right].
\end{equation}
\end{definition}

To define smoothed version correctly it is important that the function $\varphi$  is specified not only on an admissible set $\mathcal{Z}$, but in a certain neighborhood of it. This is due to the fact that for any point $z$ belonging to the set, the point $z+\tau e$ can be outside it.

\begin{lemma}[see Lemma 8 from \cite{Shamir15}]\label{lemma1}
    Let ${\varphi}(z)$ is $\mu$-strongly-convex-strongly-concave (convex-concave with $\mu = 0$) and $\mathbf{e}$ be from $\mathcal{RS}^n_2(1)$. Then function $\hat{\varphi}(z)$ is $\mu$-strongly-convex-strongly-concave  and under Assumption 2 satisfies:
    \begin{eqnarray}
        \label{lemma1_0}
        \sup_{z \in \mathcal{Z}} |\hat{\varphi}(z) - {\varphi}(z)| \leq \tau M. 
    \end{eqnarray}
\end{lemma}

\begin{lemma}[see Lemma 10 from \cite{Shamir15} and Lemma 2 from \cite{Beznosikov}]\label{lemmasham}
    Under Assumption 4 it holds that 
    \begin{eqnarray}
    \label{temp985}
    \tilde \nabla \hat{\varphi}(z) = \mathbb{E}_{\mathbf{e}}\left[\frac{n\left(\varphi(z + \tau \mathbf{e}) -  \varphi(z - \tau \mathbf{e})\right)}{2\tau}\left(
    \begin{array}{c}
    \mathbf{e}_x\\
    -\mathbf{e}_y \\
    \end{array}\right)\right],\\
    \|\mathbb{E}_{\mathbf{e},\xi}[ g(z, \mathbf{e}, \tau, \xi^{\pm}) ] - \tilde \nabla \hat{\varphi}(z)\|_q \leq \frac{\Delta n a_q}{\tau}\label{temp986}\hspace{2cm}.
    \end{eqnarray}
\end{lemma}

Now we are ready to present the main results of this section. Let begin with \textbf{convex-concave} case (Assumption 3 with $\mu=0$)
\begin{theorem} \label{th_main}
Let problem \eqref{problem} with function $\varphi(x,y)$ be solved using Algorithm \ref{alg1} with the oracle \eqref{eq:PrSt3}. Assume, that the set $\mathcal{Z}$, the convex-concave function  $\varphi(x,y)$ and its inexact modification $\widetilde{\varphi}(x,y)$ satisfy Assumptions 1, 2, 4. Denote by $N$ the number of iterations and $\gamma_k = \gamma =const$. Then the rate of convergence is given by the following expression:
\begin{eqnarray*}
    \mathbb{E}\left[\varepsilon_{sad}(\bar z_{N})\right] &\leq&
    \frac{3\Omega^2}{2 \gamma (N+1)} + \frac{3\gamma M^2_{all}}{2} + \frac{\Delta \Omega n a_q}{\tau} + 2\tau M.
\end{eqnarray*} 
$\Omega$ is a diameter of $\mathcal{Z}$, $M^2_{all} = 3\left(3n M^2 +\frac{n^2(\sigma^2 + \Delta^2)}{\tau^2}\right)a^2_q$ and 
\begin{equation}
    \label{sad}
    \varepsilon_{sad}(\bar z_{N}) = \max_{y' \in \mathcal{Y}} \varphi(\bar x_{N}, y') - \min_{x' \in \mathcal{X}} \varphi(x', \bar y_{N}).
\end{equation}
\end{theorem}

Let analyze the results:
\begin{corollary} Under the assumptions of the Theorem 1 let $\varepsilon$ be accuracy of the solution of the problem \eqref{problem} obtained using Algorithm \ref{alg1}. Assume that
\begin{eqnarray}
    \label{temp1209}
    \gamma = \Theta \left(\frac{\Omega}{n^{\frac{1}{4} + \frac{1}{2q}} M N^{\frac{3}{4}}}\right), \quad \tau = \Theta \left( \frac{\sigma}{M} \cdot \frac{n^{\frac{1}{4} + \frac{1}{2q}}}{N^{\frac{1}{4}}}\right),\quad \Delta = \mathcal{O} \left(\frac{\varepsilon \tau}{\Omega n a_q}\right),
\end{eqnarray}
then the number of iterations to find $\varepsilon$-solution
\begin{eqnarray*}
    N = \mathcal{O} \left( \frac{n^{1 + \frac{2}{q}}}{\varepsilon^4} \left[C^4(n,q) M^4 \Omega^4 +  \sigma^4\right]\right),
\end{eqnarray*}
or with
\begin{eqnarray*}
    \gamma = \Theta \left(\frac{\Omega}{n^{\frac{1}{q}} M N^{\frac{3}{4}}}\right), \quad \tau = \Theta \left( \frac{\sigma}{M} \cdot \frac{n^{\frac{1}{2}}}{N^{\frac{1}{4}}}\right),\quad \Delta = \mathcal{O} \left(\frac{\varepsilon \tau}{\Omega n a_q}\right),
\end{eqnarray*}
\begin{eqnarray*}
    N = \mathcal{O} \left( \frac{n^{\frac{4}{q}} C^4(n,q)}{\varepsilon^4}M^4 \Omega^4 +  \frac{n^{2}}{\varepsilon^4} \sigma^4  \right),
\end{eqnarray*}
where $C(n,q) \eqdef \min\{2q - 1, 32 \log n - 8\}$.
\end{corollary}

Analyse separately cases with $p = 1$ and $p = 2$.
\begin{table}[H]
    \centering
    \begin{tabular}{ |c | c | c |  }
    \hline
    $p$, ($1\leqslant p \leqslant 2$) & $q$,  ($2\leqslant q \leqslant \infty$) &   $N$, Number of iterations\\ \hline
    $p = 2$& $q = 2$ & $\mathcal{O}\left(n^{2} \varepsilon^{-4}\right)$\\\hline
    $p = 1$& $q = \infty$  & $\mathcal{O}\left(n \log^4 n \cdot \varepsilon^{-4}\right)$\\ \hline
\end{tabular}
    \caption{Summary of  convergence estimation for non-smooth case: $p = 2$ and $p = 1$.}
    \label{summary_estim}
\end{table}

Next we consider \textbf{$\mu$-strongly-convex-strongly-concave}.  Here we work with $V_z(w) = \frac{1}{2}\|z-w\|^2_2$.

\begin{theorem} \label{th_main1}
Let problem \eqref{problem} with function $\varphi(x,y)$ be solved using Algorithm \ref{alg1} with $V_z(w) = \frac{1}{2}\|z-w\|^2_2$ and the oracle \eqref{eq:PrSt3}. Assume, that the set $\mathcal{Z}$, the function  $\varphi(x,y)$ and its inexact modification $\widetilde{\varphi}(x,y)$ satisfy Assumptions 1, 2, 3, 4. Denote by $N$ the number of iterations and $\gamma_k = \frac{1}{\mu k}$. Then the rate of convergence is given by the following expression:
\begin{eqnarray*}
\EE\left[\varphi(\bar x_N, y^*) - \varphi(x^*, \bar y_N)\right]
    &\leq& \frac{M^2_{all} \log  (N+1)}{2 \mu (N+1)} + \frac{\Delta n\Omega}{\tau} + 2\tau M
\end{eqnarray*}
$\Omega$ is a diameter of $\mathcal{Z}$, $M^2_{all} = 3\left(3n M^2 +\frac{n^2(\sigma^2 + \Delta^2)}{\tau^2}\right)$.
\end{theorem}
From here one can get 
\begin{corollary} Under the assumptions of the Theorem 2 let $\varepsilon$ be accuracy of the solution of the problem \eqref{problem} obtained using Algorithm \ref{alg1}. Assume that
\begin{eqnarray*}
    \tau = \Theta \left( \sqrt[3]{\frac{\sigma^2}{\mu M}} \cdot \sqrt[3]{\frac{n^2}{N}}\right),\quad \Delta = \mathcal{O} \left(\frac{\varepsilon \tau}{\Omega n}\right),
\end{eqnarray*}
then the number of iterations to find $\varepsilon$-solution
\begin{eqnarray*}
    N = \mathcal{\tilde O} \left( \frac{ n M^2 }{\mu\varepsilon}  + \frac{M^2 n^2 \sigma^2}{\mu \varepsilon^3}\right).
\end{eqnarray*}
\end{corollary}

\textbf{Random euclidean direction with residual feedback.} In this part of the work we use the technique from \cite{zhang2020improving}. In more detail, in Algorithm \ref{alg1} we replace $g(z_k, \mathbf{e}_k, \tau,  \xi^{\pm}_k)$ with 
\begin{eqnarray}
    \label{eq:PrSt5}
    \tilde g(z_k, z_{k-1},  \mathbf{e}_k, \mathbf{e}_{k-1},  \xi_k, \xi_{k-1}) \hspace{5cm}\nonumber\\
    = \frac{n \left(\tilde \varphi(z_k + \tau \mathbf{e}_k,  \xi_k) -  \tilde \varphi(z_{k-1} + \tau \mathbf{e}_{k-1},  \xi_{k-1})\right)}{\tau}\left(
    \begin{array}{c}
    (\mathbf{e}_k)_x\\
    -(\mathbf{e}_k)_y \\
    \end{array}
    \right). 
\end{eqnarray}
The main advantage of this technique is that it requires only one call to the oracle per iteration. 

We consider only convex-concave case in the Eulidean setup, i.e. $V_z(w) = \frac{1}{2}\|z - w \|_2^2$. Let us carry out reasoning similar to the analysis of Theorem \ref{th_main}.

\begin{lemma}
    \label{beznos1}
    For $\tilde g_k \eqdef \tilde g(z_k, z_{k-1},  \mathbf{e}_k, \mathbf{e}_{k-1},  \xi_k, \xi_{k-1})$ defined in \eqref{eq:PrSt5} under Assumptions 2 and 4 the following inequalities holds:
    \begin{eqnarray}
    \label{boundlem1}
    \mathbb{E}\left[\|\tilde g_k\|^2_2\right] 
    &\leq& \alpha^k\mathbb{E}\left[\|\tilde g_{0}\|^2_2 \right] + \left(\frac{12n^2 (\sigma^2 + \Delta^2)}{\tau^2}+ 12n^2M^2\right) \frac{1}{1 - \alpha},
    \end{eqnarray}
     where  $\alpha = \frac{6\gamma^2 n^2M^2}{\tau^2} < 1$.
\end{lemma}

\begin{lemma}
    Under Assumption 4 it holds that 
    \begin{eqnarray}
    \label{temp9851}
    \tilde \nabla \hat{\varphi}(z_k) = \mathbb{E}_{\mathbf{e}_k}\left[\frac{n\left(\varphi(z_k + \tau \mathbf{e}_k) -  \varphi(z + \tau \mathbf{e}_{k-1})\right)}{\tau}\left(
    \begin{array}{c}
    (\mathbf{e}_k)_x\\
    -(\mathbf{e}_k)_y \\
    \end{array}\right)\right],\\
    \|\mathbb{E}_{\mathbf{e}_k}[ \tilde g_k ] - \tilde \nabla \hat{\varphi}(z_k)\|_2 \leq \frac{\Delta n}{\tau} \hspace{4.5cm}\label{temp9861}.
    \end{eqnarray}
\end{lemma}

\begin{theorem} 
Let problem \eqref{problem} with function $\varphi(x,y)$ be solved using Algorithm \ref{alg1} with $V_z(w) = \frac{1}{2}\|z-w\|^2_2$ and the oracle \eqref{eq:PrSt5}. Assume, that the set $\mathcal{Z}$, the convex-concave function $\varphi(x,y)$ and its inexact modification $\widetilde{\varphi}(x,y)$ satisfy Assumptions 1, 2, 4. Denote by $N$ the number of iterations and $\gamma_k = \gamma = const$. Then the rate of convergence is given by the following expression:
\begin{eqnarray*}
    \mathbb{E}\left[\varepsilon_{sad}(\bar z_{N})\right] &\leq&
    \frac{3\Omega^2}{2\gamma(N+1)} + \frac{3\gamma}{2(N+1)(1 - \alpha)} \mathbb{E}\left[\|\tilde g_{0}\|^2_2 \right] \nonumber\\
    && + \frac{3\gamma}{2(1 - \alpha)} \left(\frac{12n^2 (\sigma^2 + \Delta^2)}{\tau^2}+ 12n^2M^2\right)  + 2\tau M  + \frac{\Delta \Omega n}{\tau}.
\end{eqnarray*} 
$\Omega$ is a diameter of $\mathcal{Z}$, $\alpha = \frac{6\gamma^2 n^2M^2}{\tau^2} < 1$.
\end{theorem}

Next we analyze the results:
\begin{corollary} Under the assumptions of the Theorem 3 let $\varepsilon$ be accuracy of the solution of the problem \eqref{problem} obtained using Algorithm \ref{alg1} with \eqref{eq:PrSt5}. Assume that
\begin{eqnarray*}
    \gamma = \left(\frac{\Omega \tau}{6n M N^{\frac{1}{2}}}\right), \quad \tau = \Theta \left( \frac{\sigma}{M}\cdot \frac{n^{\frac{1}{2}} }{N^{\frac{1}{4}}}\right),\quad \Delta = \mathcal{O} \left(\frac{\varepsilon \tau}{ \Omega n}\right),
\end{eqnarray*}
then the number of iterations to find $\varepsilon$-solution
\begin{eqnarray*}
    N = \mathcal{O} \left( \frac{n^2}{\varepsilon^4} \left[M^4 \Omega^4  +   \sigma^4\right]\right).
\end{eqnarray*}
\end{corollary}

\subsection{Smooth case}\label{32}

\begin{assumption}[Gradient's Lipschitz continuity]\label{grad_lip}
The gradient $\nabla \varphi(z)$ of the function $\varphi$ is $L$-Lipschitz continuous in certain neighbourhood of $\mathcal{Z}$ with $L > 0$ w.r.t. norm $\|\cdot\|_2$ when 
\begin{equation*}
| \nabla \varphi(z)- \nabla \varphi(z')|\leq L\|z-z'\|_2, \quad \forall\ z, z'\in \mathcal{Z}.
\end{equation*}
\end{assumption}

\begin{lemma}[see Lemma A.3 from \cite{akhavan2020exploiting}]\label{lemma_ltau^2}
    Let ${\varphi}(z)$ be convex-concave (or $\mu$-strongly-convex-strongly-concave) and $\mathbf{e}$ be from $\mathcal{RS}^n_2(1)$. Then function $\hat{\varphi}(z)$ is convex-concave ($\mu$-strongly-convex-strongly-concave) too  and under Assumption \ref{grad_lip} satisfies:
    \begin{eqnarray}
        \label{Ltau^2}
        \sup_{z \in \mathcal{Z}} |\hat{\varphi}(z) - {\varphi}(z)| \leq \dfrac{L\tau^2}{2}. 
    \end{eqnarray}
\end{lemma}

\begin{theorem} \label{th_smooth}
Let problem \eqref{problem} with function $\varphi(x,y)$ be solved using Algorithm \ref{alg1} with the oracle \eqref{eq:PrSt3}. Assume, that the set $\mathcal{Z}$, the convex-concave function  $\varphi(x,y)$ and its inexact modification $\widetilde{\varphi}(x,y)$ satisfy Assumptions 1,4,5. Denote by $N$ the number of iterations and $\gamma_k = \gamma =const$. Then the rate of convergence is given by the following expression:
\begin{eqnarray*}
    \mathbb{E}\left[\varepsilon_{sad}(\bar z_{N})\right] &\leq&
    \frac{3\Omega^2}{2 \gamma (N+1)} + \frac{3\gamma M^2_{all}}{2} + \frac{\Delta \Omega n a_q}{\tau} + L\tau^2.
\end{eqnarray*} 
$\Omega$ is a diameter of $\mathcal{Z}$, $M^2_{all} = 3\left(3n M^2 +\frac{n^2(\sigma^2 + \Delta^2)}{\tau^2}\right)a^2_q$.
\end{theorem}

Let's analyze the results:
\begin{corollary} Under the assumptions of the Theorem \ref{th_smooth} let $\varepsilon$ be accuracy of the solution of the problem \eqref{problem} obtained using Algorithm \ref{alg1}. Assume that
\begin{eqnarray}
    \gamma = \Theta \left(\frac{\Omega}{n^{\frac{1}{3} + \frac{2}{3q}} M N^{\frac{2}{3}}}\right), \quad \tau = \Theta \left( \frac{\sigma}{M} \cdot \frac{n^{\frac{1}{6} + \frac{1}{3q}}}{N^{\frac{1}{6}}}\right),\quad \Delta = \mathcal{O} \left(\frac{\varepsilon \tau}{\Omega n a_q}\right),
\end{eqnarray}
then the number of iterations to find $\varepsilon$-solution
\begin{eqnarray*}
    N = \mathcal{O} \left( \frac{n^{1 + \frac{2}{q}}}{\varepsilon^3} \left[M^3\Omega^3 + \frac{L^3 \sigma^3}{M^3}\right]\right).
\end{eqnarray*}
\end{corollary}

\begin{theorem}\label{th_smooth_strong}
Let problem \eqref{problem} with function $\varphi(x,y)$ be solved using Algorithm \ref{alg1} with $V_z(w) = \frac{1}{2}\|z-w\|^2_2$ and the oracle \eqref{eq:PrSt3}. Assume, that the set $\mathcal{Z}$, the function  $\varphi(x,y)$ and its inexact modification $\widetilde{\varphi}(x,y)$ satisfy Assumptions 1, 3, 4, 5. Denote by $N$ the number of iterations and $\gamma_k = \frac{1}{\mu k}$. Then the rate of convergence is given by the following expression:
\begin{eqnarray*}
\EE\left[\varphi(\bar x_N, y^*) - \varphi(x^*, \bar y_N)\right]
    &\leq& \frac{M^2_{all} \log (N+1)}{2 \mu (N+1)} + \frac{\Delta n\Omega}{\tau} + L\tau^2.
\end{eqnarray*}
$\Omega$ is a diameter of $\mathcal{Z}$, $M^2_{all} = 3\left(3n M^2 +\frac{n^2(\sigma^2 + \Delta^2)}{\tau^2}\right)$.
\end{theorem}

Let's analyze the results:
\begin{corollary} Under the assumptions of the Theorem \ref{th_smooth_strong} let $\varepsilon$ be accuracy of the solution of the problem \eqref{problem} obtained using Algorithm \ref{alg1}. Assume that
\begin{eqnarray*}
    \tau = \Theta \left( \sqrt[4]{\frac{\sigma^2}{\mu L}} \cdot \frac{n^{\frac{1}{2}}}{N^{\frac{1}{4}}}\right),\quad \Delta = \mathcal{O} \left(\frac{\varepsilon \tau}{\Omega n}\right),
\end{eqnarray*}
then the number of iterations to find $\varepsilon$-solution
\begin{eqnarray*}
    N = \widetilde{\mathcal{O}} \left( \frac{n M^2}{\mu \varepsilon} + \frac{Ln ^2 \sigma^2}{\mu \varepsilon^2} \right).
\end{eqnarray*}
\end{corollary}

\subsection{Higher-order smooth case}\label{33}

In this paragraph we study higher-order smooth functions $\varphi$ functions satisfying so called generalized Hölder condition with parameter $\beta > 2$ (see inequality \eqref{Hölder-condition} below). 

\subsubsection{Higher order smoothness}

Let $l$ denote maximal integer number strictly less than $\beta$. Let ${\cal F}_{\beta}(L_\beta)$ denote the set of all functions $\varphi: \mathbb{R}^n \rightarrow \mathbb{R}$ which are differentiable $l$ times and for all $z,z_0\in U_{\varepsilon_0} (\mathcal{Z})$ satisfy Hölder condition:
\begin{equation}\label{Hölder-condition}
    \Biggl| \varphi(z) - \sum_{0\leq|m|\leq l} \dfrac{1}{m!} D^m \varphi(z_0) (z-z_0)^m \Biggr| \leq L_\beta\|z-z_0\|^{\beta},
\end{equation}
where $L_\beta>0$, the sum is over multi-index $m=(m_1, \dots, m_n) \in \mathbb{N}^n$, we use the notation $m! = m_1! \cdot \dots \cdot m_n!$, $|m| = m_1 + \dots + m_n$ and we defined
\begin{equation*}
     D^m \varphi(z_0) z^m = \dfrac{\partial^{|m|} \varphi(z_0)}{\partial^{m_1} z_1  \dots \partial^{m^n} z_n} z_1^{m_1} \cdot \dots \cdot z_n^{m_n}, \; \forall z=(z_1, \dots, z_n) \in \R^n.
\end{equation*}

Let ${\cal F}_{\mu, \beta}(L_\beta)$ denote the set of $\mu$-strongly-convex-strongly-concave functions $\varphi \in {\cal F}_{\beta}(L_\beta)$.

To use the higher-order smoothness we propose smoothing kernel though this is not the only way. We propose to use Algorithm \ref{algo1} which uses the kernel smoothing technique.
In fact the Algorithm \ref{algo1} arises from Algorithm \ref{alg1} in the Euclidean setting ($V_z(w) = \frac{1}{2} \| z-w\|_2^2$).

\begin{algorithm}[H]
\caption{Zero-order Stochastic Projected Gradient} \label{algo1}
\begin{algorithmic}
\State 
\noindent {\bf Requires: } Kernel $K: [-1, 1] \rightarrow \mathbb{R}$, step size $\gamma_k>0$, parameters $\tau_k$.

\State
{\bf Initialization: } Generate scalars $r_1, \dots, r_N$ uniformly on $[-1,1]$ and vectors $e_1, \dots, e_N$ uniformly on the Euclidean unit sphere $S_n=\{e\in \mathbb{R}^n: \, \|e\|=1 \}$.
\For{$k=1, \dots, N$}{
 \begin{enumerate}
     \item $\widetilde{\varphi}_k^+ := \varphi(z_k+\tau_k r_k e_k) + \xi^+_k$, $\widetilde{\varphi}_k^- := \varphi(z_k-\tau_k r_k e_k) + \xi^-_k$
     \item Define $\widetilde{g_k}:= \frac{n}{2\tau_k}(\widetilde{\varphi}_k^+-\widetilde{\varphi}_k^-)\left(
    \begin{array}{c}
    (\mathbf{e}_k)_x\\
    -(\mathbf{e}_k)_y \\
    \end{array}
    \right)K(r_k)$
     \item Update $z_{k+1} := \Pi_{Q}(z_k-\gamma_k\widetilde{g_k})$
 \end{enumerate}
}
\EndFor
\State 
\noindent {\bf Output:} $\left\{z_k\right\}_{k=1}^N$.
\end{algorithmic}
\end{algorithm}

To use the higher-order smoothness we propose we need to introduce additional noise assumption:

\begin{assumption}\label{noise-ass}
    For all $k = 1, 2, \dots, N$ it holds that
    \begin{enumerate}
        \item\label{i} $\E[\xi_k^{+2}] \leq \sigma^2$ and $\E[\xi_k^{-2}] \leq \sigma^2$ where $\sigma \geq 0$;
        \item\label{ii} the random variables $\xi^+_k$ and $\xi^-_k$ are independent from $e_k$ and $r_k$, the random variables $e_k$ and $r_k$ are independent.
    \end{enumerate}
\end{assumption} 

In other words we assume that $\delta(x,y)$ in \eqref{PrSt1} is equal to zero. 
We do not assume here neither zero-mean of $\xi^+_k$ and $\xi^-_k$ nor i.i.d of $\{\xi^+_k\}_{k=1}^{N}$ and $\{\xi^-_k\}_{k=1}^{N}$ as item \ref{ii} from Assumption \ref{noise-ass} allows to avoid that.

\subsubsection{Kernel}
For gradient estimator $\widetilde{g_k}$ we use the kernel 
\begin{equation*}
    K: [-1, 1] \rightarrow \mathbb{R},
\end{equation*}
satisfying
\begin{equation}\label{kernel-properties}
    \E[K(r)] = 0, \,
    \E[rK(r)] = 1, \, 
    \E [r^j K(r)] = 0,\, j=2, \dots, l, \,
    \E\left[ |r|^{\beta}|K(r)|\right] \leq \infty,
\end{equation}
where $r$ is a uniformly distributed on $[-1, 1]$ random variable. This helps us to get better bounds on the gradient bias $\| \widetilde{g_k} - \nabla f(x_k) \|$ (see Theorem \ref{theorem_highorder_smooth_strongly_convex} for details). The examples of possible kernels are presented in Appendix \ref{appendix:kernels}.

For Theorem \ref{theorem_highorder_smooth_strongly_convex} and Theorem \ref{theorem_highorder_smooth}  we need to introduce the constants
\begin{equation}\label{kappa-beta}
    \kappa_{\beta} = \int|u|^{\beta}|K(u)|\,du
\end{equation}
and 
\begin{equation}\label{kappa-squared}
    \kappa =\int K^2(u) \, du.
\end{equation}
It is proved in \cite{bach2016highly} that $\kappa_{\beta}$ and $\kappa$ do not depend on $n$, they depend only on $\beta$: 
\begin{equation}\label{kappa-beta-bound}
    \kappa_{\beta} \leq 2\sqrt{2} (\beta - 1),
\end{equation}
\begin{equation}\label{kappa-squared-bound}
    \kappa \leq \sqrt{3} \beta^{\nicefrac{3}{2}}.
\end{equation}

\begin{theorem} \label{theorem_highorder_smooth_strongly_convex}
Let $\varphi \in {\cal F}_{\mu, \beta}(L)$ with $\mu$, $L > 0$ and $\beta > 2$. 
Let Assumption \ref{noise-ass} hold 
and let $\mathcal{Z}$ be a convex compact subset of $\R^n$.
Let $\varphi$ be $M$-Lipschitz on the Euclidean $\tau_1$-neighborhood of $\mathcal{Z}$ (see $\tau_k$ below). 

Then the rate of convergence is given by Algorithm \ref{algo1} with parameters
\begin{equation*}
    \tau_k = \left(\dfrac{3\kappa\sigma^2n}{2(\beta-1)(\kappa_\beta L)^2}\right)^{\frac{1}{2\beta}}k^{-\frac{1}{2\beta}}, \quad \alpha_k=\dfrac{2}{\mu k}, \quad k =1,\dots, N
\end{equation*}
satisfies 
\begin{equation*}
\begin{split}
    \E \left[ \varphi(\overline{x}_N, y^*) - \varphi(x^*, \overline{y}_N) \right] &\leq \max_{y\in \mathcal{Y}} \E \left[\varphi(\overline{x}_N, y)  \right] -
    \min_{x\in \mathcal{X}}\E \left[\varphi(x, \overline{y}_N) \right] \\
    &\leq \dfrac{1}{\mu} \left(n^{2-\frac{1}{\beta}}\dfrac{A_1}{N^{\frac{\beta-1}{\beta}}}+A_2\dfrac{n(1+\ln N)}{N} \right),
\end{split}
\end{equation*}
where $\overline{z}_N = \frac{1}{N}\sum\limits_{k=1}^N z_k$, $A_1=3\beta(\kappa \sigma^2)^{\frac{\beta-1}{\beta}}(\kappa_{\beta}L)^{\frac{2}{\beta}}$, $A_2=9\kappa G^2$, $\kappa_{\beta}$ and $\kappa$ are constants depending only on $\beta$, see \eqref{kappa-beta} and \eqref{kappa-squared}.
\end{theorem}

We emphasize that the usage of kernel smoothing technique, measure concentration inequalities and the assumption that $\xi_k$ is independent from $e_k$ or $r_k$ (Assumption \ref{noise-ass}) lead to the results better  than the state-of-the-art ones for $\beta > 2$. The last assumption also allows us not to assume neither zero-mean of $\xi^+_k$ and $\xi^-_k$ nor i.i.d of $\{\xi^+_k\}_{k=1}^{N}$ and $\{\xi^-_k\}_{k=1}^{N}$.

\begin{theorem} \label{theorem_highorder_smooth}

Let $\varphi \in {\cal F}_{\beta}(L)$ with $L > 0$ and $\beta > 2$. 
Let Assumption \ref{noise-ass} hold 
and let $\cal{Z}$ be a convex compact subset of $\R^n$.
Let $\varphi$ be $M$-Lipschitz on the Euclidean $\tau_1$-neighborhood of $\cal{Z}$ ($\tau_k$ is parameter from Theorem \ref{theorem_highorder_smooth_strongly_convex} for the regularized function $\varphi_\mu(z)$ whose  description is given below). Let $\overline{z}_N$ denote $\frac{1}{N}\sum\limits_{k=1}^N z_k$.

Let's define $N(\varepsilon)$:

\begin{equation*}
    N(\varepsilon)=\max\left\{ \left(R\sqrt{2A_1}\right)^{\frac{2\beta}{\beta-1}}\dfrac{n^{2+\frac{1}{\beta-1}}}{\varepsilon^{2+\frac{2}{\beta-1}}},\left(R\sqrt{2c'A_2}\right)^{2(1+\rho)}\dfrac{n^{1+\rho}}{\varepsilon^{2(1+\rho)}}\right\},
\end{equation*}
where  $A_1=3\beta(\kappa \sigma^2)^{\frac{\beta-1}{\beta}}(\kappa_{\beta}L)^{\frac{2}{\beta}}$, $A_2=9\kappa G^2$ -- constants from Theorem \ref{theorem_highorder_smooth_strongly_convex}, $\rho > 0$ -- arbitrarily small positive number, $c'$ -- constant which depends on $\rho$. 

Then the rate of convergence is given by the following expression:
\begin{equation}
    \E \left[ \varphi(\overline{x}_N, y^*) - \varphi(x^*, \overline{y}_N) \right] \leq \max_{y\in \mathcal{Y}} \E \left[\varphi(\overline{x}_N, y)  \right] -
    \min_{x\in \mathcal{X}}\E \left[\varphi(x, \overline{y}_N) \right] \leq \varepsilon
\end{equation}
after $N(\varepsilon)$ steps of Algorithm \ref{algo1} with settings from Theorem \ref{theorem_highorder_smooth_strongly_convex} for the regularized function: $\varphi_{\mu}(z):=\varphi(z)+\frac{\mu}{2} \| x-x_0 \|^2-\frac{\mu}{2} \| y - y_0 \|^2$, where $\mu \leq \frac{\varepsilon}{R^2}$, $R=\|z_0-z^*\|$, $z_0 \in \mathcal{Z}$ -- arbitrary point.

\end{theorem}


\section{Experiments}

In our experiments we consider the classical bilinear problem on a probability simplex:

\begin{eqnarray}
\label{exp_pr_4}
\min_{x\in \Delta_n}\max_{y\in \Delta_k} \left[  y^T Cx\right],
\end{eqnarray}

This problem has many different applications and interpretations, one of the main ones is a matrix game (see Part 5 in \cite{nemirovski}), i.e. the element $c_{ij}$ of the matrix are interpreted as a winning, provided that player $X$ has chosen the $i$th strategy and player $Y$ has chosen the $j$th strategy, the task of one of the players is to maximize the gain, and the opponent’s task -- to minimize.

The step of our algorithms can be written as follows (see \cite{aleks2020gradientfree}):
\begin{eqnarray*}
[x_{k+1}]_i = \frac{[x_k]_i \exp(-\gamma_k [g_x]_i)}{\sum\limits_{j=1}^n [x_k]_j \exp(-\gamma_k [g_x]_j)}, ~~~~
[y_{k+1}]_i = \frac{[y_k]_i \exp(\gamma_k [g_y]_i)}{\sum\limits_{j=1}^n [y_k]_j \exp(\gamma_k [g_y]_j)},
\end{eqnarray*}
where under $g_x, g_y$ we mean parts of $g$ which are responsible for $x$ and for $y$. Note that we do not present a generalization of Algorithm \ref{algo1} in an arbitrary Bregman setup, but we want to check in practice. 

We take matrix $50 \times 50$. All elements of the matrix are generated from the uniform distribution from 0 to 1. Next, we select one row of the matrix and generate its elements from the uniform from 5 to 10. Finally, we take one element from this row and generate it uniformly from 1 to 5. Finally, the matrix is normalized. Further, with each call of the function value $y^T Cx$ we add stochastic noise with constant variance (which is on average 5\% or 10\% of the function value). 

The main goal of our experiments is to compare three gradient-free approaches: Algorithm \ref{alg1} with \eqref{eq:PrSt3} and \eqref{eq:PrSt5} approximations, as well as Algorithm \ref{algo1}. We also added a first order method for comparison. Parameters $\gamma$ and $\tau$ are selected with the help of grid-search so that the convergence is the fastest, but stable. 
See Figure \ref{fig:1} for results.
\begin{figure}[h!]
\centering
\begin{minipage}{0.48\textwidth}
\includegraphics[width =  \textwidth]{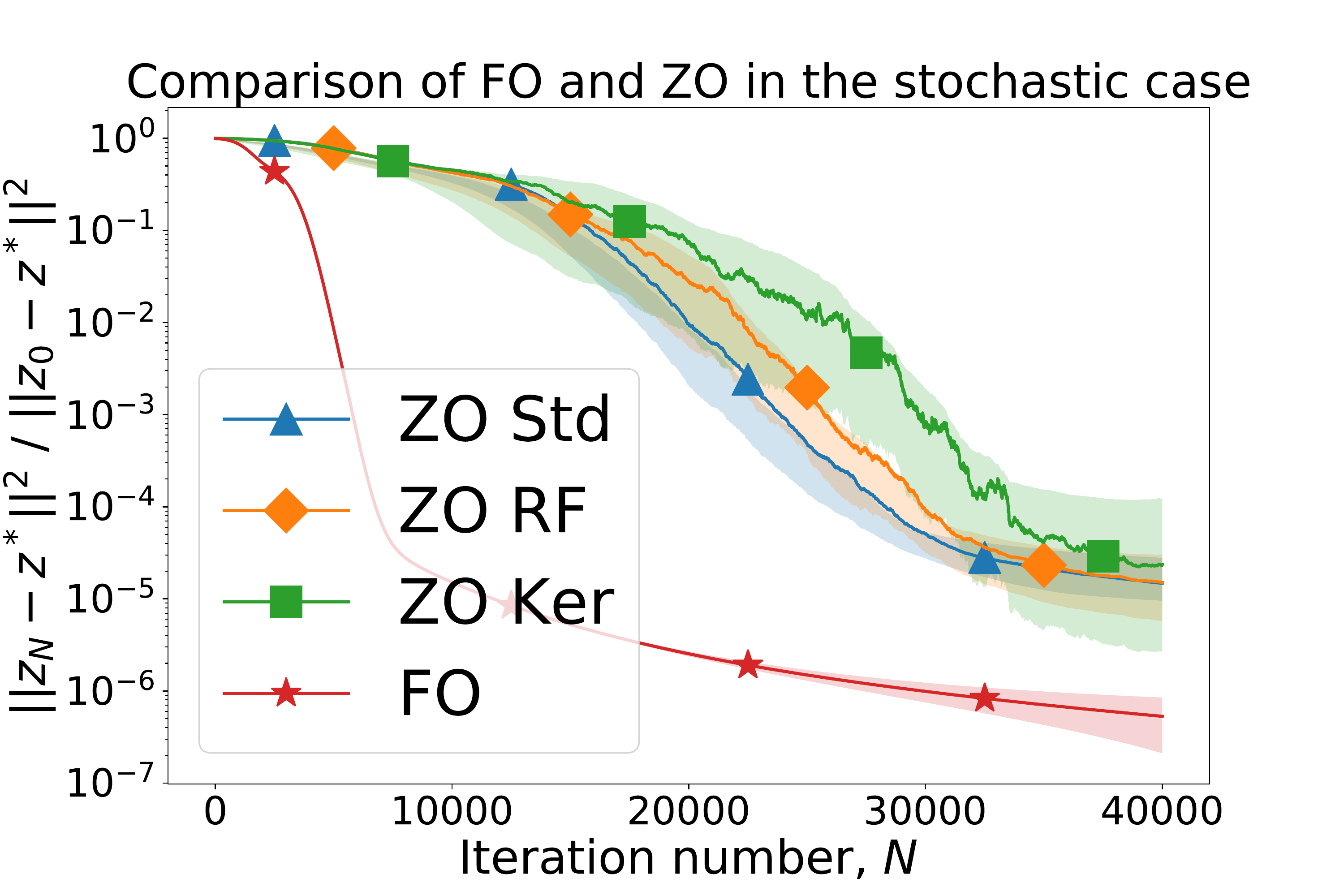}
\end{minipage}%
\begin{minipage}{0.48\textwidth}
\includegraphics[width =  \textwidth]{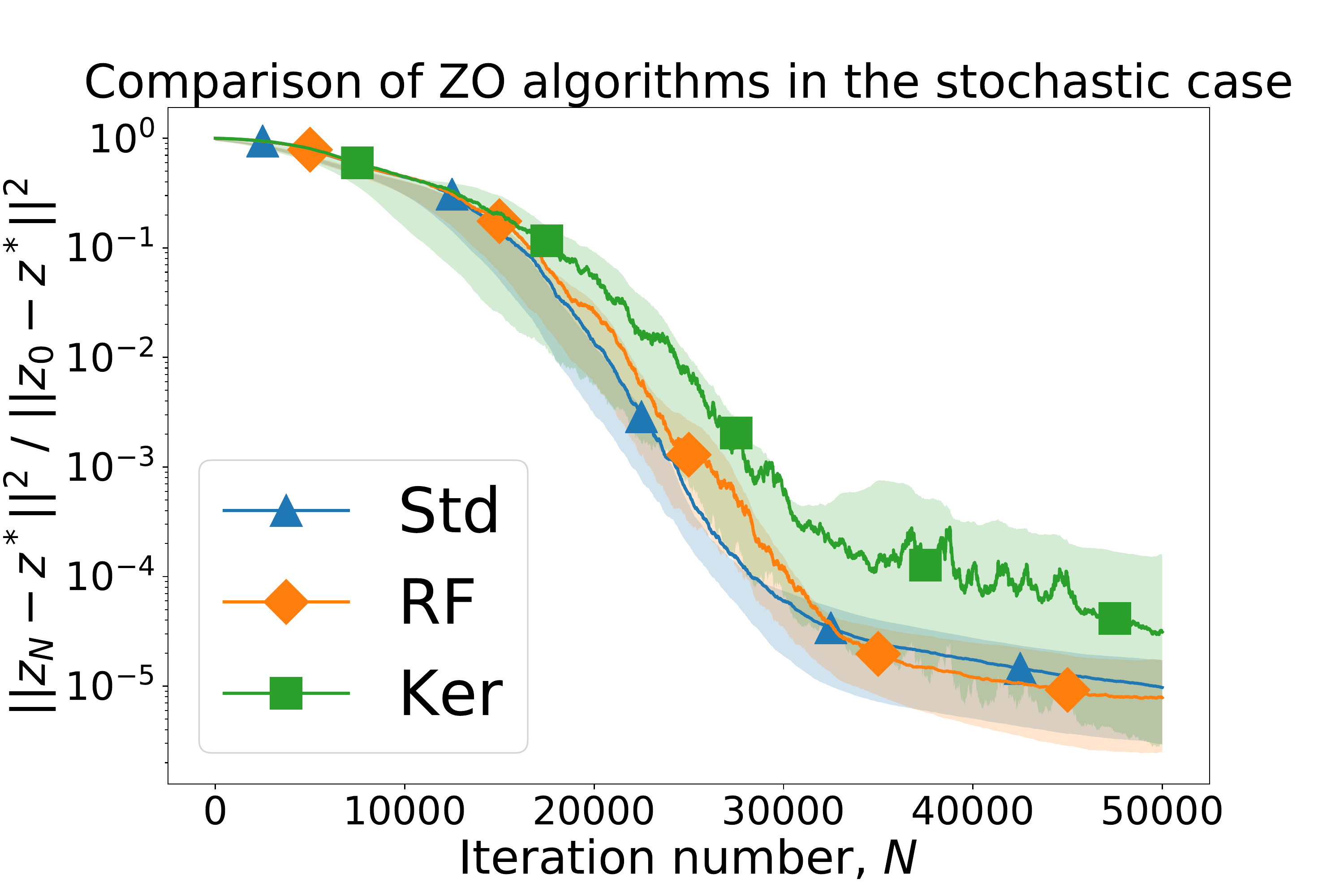}
\end{minipage}%
\\
\begin{minipage}{0.5\textwidth}
\centering
~~~(a) noise 5\%
\end{minipage}%
\begin{minipage}{0.5\textwidth}
\centering
~~~~~(b) noise 10\%
\end{minipage}%
\caption{Algorithm \ref{alg1} with \eqref{eq:PrSt3} (ZO Std) and \eqref{eq:PrSt5} (ZO RF) approximations, Algorithm \ref{algo1} (ZO Ker) and Mirror Descent (FO) applied to solve saddle-problem \eqref{exp_pr_4} with noise level: (a) 5\%, (b) 10\%.}
\label{fig:1}
\end{figure}

Based on the results of the experiments, we note that the gradient-free methods converge more slowly than the first-order method -- which is predictable. The convergence of zeroth-order methods is approximately the same, the only thing that can be noted is that the method with a kernel is subject to larger fluctuations.

\bibliographystyle{splncs04}
\bibliography{ltr}

\appendix

\section{General facts}

\begin{lemma}[see inequality 5.3.18 from \cite{nemirovski}]
Let $d(z): \mathcal{Z} \to \mathbb{R}$ is prox-function and $V_z(w)$ define Bregman divergence  associated with $d(z)$.
The following equation holds for $x,y,u \in X$:
\begin{eqnarray}
\label{temp111}
\langle \nabla d(x) - \nabla d(y), u -x \rangle = V_y(u)-V_x(u)-V_y(x).
\end{eqnarray}
\end{lemma}

\begin{lemma}[Fact 5.3.2 from \cite{nemirovski}]
    Given norm $\|\cdot \|$ on space $\mathcal{Z}$ and prox-function $d(z)$, let $z \in \mathcal{Z}$, $w \in \mathbb{R}^n$ and $z_{+} = \text{prox}_z(w)$. Then for all $u \in \mathcal{Z}$
    \begin{eqnarray}
        \label{lemma3_1}
         \langle w, z_{+} - u\rangle \leqslant V_{z}(u) - V_{z_{+}}(u) - V_{z}(z_{+}).
    \end{eqnarray}
\end{lemma}

\begin{lemma}
For arbitrary integer $n\ge 1$ and arbitrary set of positive numbers $a_1,\ldots,a_n$ we have
\begin{equation}
    \left(\sum\limits_{i=1}^m a_i\right)^2 \le m\sum\limits_{i=1}^m a_i^2.\label{eq:squared_sum}
\end{equation}
\end{lemma}

\begin{lemma}[Lemma 9 from \cite{Shamir15}]\label{lem:lemma_9_shamir} For any function $g$ which  is $M$-Lipschitz with respect to the $\ell_2$-norm, it holds that if $e$ is uniformly distributed on the Euclidean unit sphere, then 
\begin{equation*}
    \sqrt{\mathbb{E}[(g(e) - \mathbb{E}g(e))^4]} \leq  \frac{3M^2}{n}.
\end{equation*}
\end{lemma}

\section{Proofs for Section \ref{31}}

\begin{lemma} 
    For $g(z, \mathbf{e}, \tau,  \xi^{\pm})$ defined in \eqref{eq:PrSt3} under Assumptions 2 and 4 the following inequality holds:
    \begin{eqnarray*}
        \mathbb{E}\left[ \|g(z, \mathbf{e}, \tau,  \xi^{\pm})\|^2_q\right] \leq 3a_q^2 \left(3nM^2+ \frac{n^2(\sigma^2 + \Delta^2)}{\tau^2} \right),
    \end{eqnarray*}
    where $a^2_q$ is determined by $\EE[\|e\|_q^2] \leq \sqrt{\EE[\|e\|_q^4]} \le a^2_q$ and the following statement is true
    \begin{eqnarray*}
    a_q^2 = \min\{2q - 1, 32 \log n - 8\} n^{\frac{2}{q} - 1}, \quad \forall n \geq 3.
\end{eqnarray*}
\end{lemma}

\begin{proof} Using a simple fact \eqref{eq:squared_sum}, we obtain the following inequalities:
    \begin{eqnarray*}
  \mathbb{E}\left[ \|g(z, \mathbf{e}, \tau,  \xi^{\pm})\|^2_q\right]
  &=&\mathbb{E}\left[\left\|\frac{n}{2\tau}\left(\tilde{\varphi}(z+ \tau \mathbf{e},  \xi^{+}) - \widetilde{\varphi}(z - \tau \mathbf{e},  \xi^{-})\right)\mathbf{e}\right\|_q^2\right] \nonumber\\
  &=& \mathbb{E}\left[\left\|\frac{n}{2\tau}\left(\varphi(z + \tau \mathbf{e}) + \xi^+ + \delta(z + \tau \mathbf{e})  -  \varphi(z - \tau \mathbf{e}) - \xi^- - \delta(z - \tau \mathbf{e}) \right)\mathbf{e}\right\|_q^2\right] \nonumber\\
  &\leq& \frac{3n^2}{4\tau^2}\mathbb{E}\left[\left\|\left(\varphi(z + \tau \mathbf{e}) -  \varphi(z - \tau \mathbf{e})\right)\mathbf{e}\right\|_q^2\right] + \frac{3n^2}{4\tau^2}\mathbb{E}\left[\left\|\left(\xi^+ - \xi^-\right)\mathbf{e}\right\|_q^2\right] \nonumber\\
  &&+ \frac{3n^2}{4\tau^2}\mathbb{E}\left[\left\|\left(\delta(z + \tau \mathbf{e})  - \delta(z - \tau \mathbf{e}) \right)\mathbf{e}\right\|_q^2\right] \nonumber\\
  &\leq& \frac{3n^2}{4\tau^2}\mathbb{E}\left[\left(\varphi(z + \tau \mathbf{e},\xi) -  \varphi(z - \tau \mathbf{e},\xi)\right)^2\left\|\mathbf{e}\right\|_q^2\right] + \frac{3n^2}{2\tau^2}\mathbb{E}\left[\left((\xi^+)^2 + (\xi^-)^2\right)\left\|\mathbf{e}\right\|_q^2\right] \nonumber\\
  &&+ \frac{3n^2}{2\tau^2}\mathbb{E}\left[\left((\delta(z + \tau \mathbf{e}) )^2+ (\delta(z - \tau \mathbf{e}) )^2\right)\left\|\mathbf{e}\right\|_q^2\right].
  \end{eqnarray*}
By independence of $\xi^{\pm}$ and $\mathbf{e}$, we have
\begin{eqnarray*}
  \mathbb{E}\left[ \|g(z, \mathbf{e}, \tau, \xi^{\pm})\|^2_q\right]
  &\leq& \frac{3n^2}{4\tau^2}\mathbb{E}_{\xi}\left[\mathbb{E}_{\mathbf{e}}\left[\left(\varphi(z + \tau \mathbf{e}) - \alpha -  \varphi(z - \tau \mathbf{e}) + \alpha\right)^2\left\|\mathbf{e}\right\|_q^2\right]\right] \nonumber\\
  &&+ \frac{3n^2}{2\tau^2}\mathbb{E}_{\mathbf{e}}\left[\mathbb{E}_{\mathbf{\xi}}\left[\left((\xi^+)^2 + (\xi^-)^2\right)\left\|\mathbf{e}\right\|_q^2\right]\right] \nonumber\\
  &&+ \frac{3n^2}{2\tau^2}\mathbb{E}\left[\left((\delta(z + \tau \mathbf{e}))^2+ (\delta(z - \tau \mathbf{e}))^2\right)\left\|\mathbf{e}\right\|_q^2\right] \nonumber\\
  &\leq& \frac{3n^2}{2\tau^2}\mathbb{E}_{\xi}\left[\mathbb{E}_{\mathbf{e}}\left[\left(\left(\varphi(z + \tau \mathbf{e}) - \alpha\right)^2 + \left( \varphi(z - \tau \mathbf{e}) - \alpha\right)^2 \right)\left\|\mathbf{e}\right\|_q^2\right]\right] \nonumber\\
  &&+ \frac{3n^2}{2\tau^2}\mathbb{E}_{\mathbf{e}}\left[\mathbb{E}_{\mathbf{\xi}}\left[\left((\xi^+)^2 + (\xi^-)^2\right)\left\|\mathbf{e}\right\|_q^2\right]\right] \nonumber\\
  &&+ \frac{3n^2}{2\tau^2}\mathbb{E}\left[\left((\delta(z + \tau \mathbf{e}))^2+ (\delta(z - \tau \mathbf{e}))^2\right)\left\|\mathbf{e}\right\|_q^2\right] .
  \end{eqnarray*}
Taking into account the symmetric distribution of $\mathbf{e}$  and Cauchy--Schwarz inequality:
\begin{eqnarray*}
  \mathbb{E}\left[ \|g(z, \mathbf{e}, \tau,  \xi^{\pm})\|^2_q\right]
  &\leq& \frac{3n^2}{\tau^2}\mathbb{E}_{\xi}\left[\mathbb{E}_{\mathbf{e}}\left[\left(\varphi(z + \tau \mathbf{e}) - \alpha\right)^2 \left\|\mathbf{e}\right\|_q^2\right]\right] + \frac{3n^2}{2\tau^2}\mathbb{E}_{\mathbf{e}}\left[\mathbb{E}_{\mathbf{\xi}}\left[\left((\xi^+)^2 + (\xi^-)^2\right)\left\|\mathbf{e}\right\|_q^2\right]\right] \nonumber\\
  &&+ \frac{3n^2}{2\tau^2}\mathbb{E}\left[\left((\delta(z + \tau \mathbf{e}))^2+ (\delta(z - \tau \mathbf{e}))^2\right)\left\|\mathbf{e}\right\|_q^2\right] \nonumber\\
  &\leq& \frac{3n^2}{\tau^2}\mathbb{E}_{\xi}\left[\sqrt{\mathbb{E}_{\mathbf{e}}\left[\left(\varphi(z + \tau \mathbf{e},\xi) - \alpha\right)^4\right]} \sqrt{\mathbb{E}_{\mathbf{e}}\left[\left\|\mathbf{e}\right\|_q^4\right]}\right] \nonumber\\
  &&+ \frac{3n^2}{2\tau^2}\mathbb{E}_{\mathbf{e}}\left[\mathbb{E}_{\mathbf{\xi}}\left[\left((\xi^+)^2 + (\xi^-)^2\right)\left\|\mathbf{e}\right\|_q^2\right]\right] \nonumber\\
  &&+ \frac{3n^2}{2\tau^2}\mathbb{E}\left[\left((\delta(z + \tau \mathbf{e}))^2+ (\delta(z - \tau \mathbf{e}))^2\right)\left\|\mathbf{e}\right\|_q^2\right] \nonumber\\
  &\leq& \frac{3n^2 a_q^2}{\tau^2}\mathbb{E}_{\xi}\left[\sqrt{\mathbb{E}_{\mathbf{e}}\left[\left(\varphi(z + \tau \mathbf{e},\xi) - \alpha\right)^4\right]}\right] + \frac{3n^2 a_q^2 (\sigma^2 + \Delta^2)}{\tau^2}.
  \end{eqnarray*}  
In the last inequalities we use \eqref{eq:PrSt2} and \eqref{eq:condition_on_u}. Substituting $\alpha = \mathbb{E}\left[\varphi(z + \tau \mathbf{e})\right]$, applying Lemma \ref{lem:lemma_9_shamir} with the fact that $\varphi(z + \tau \mathbf{e})$ is $\tau M$-Lipschitz w.r.t. $\mathbf{e}$ in terms of the $\|\cdot\|_2$-norm we get
\begin{eqnarray*}
  \mathbb{E}\left[ \|g(z, \mathbf{e}, \tau,  \xi^{\pm})\|^2_q\right]
  &\leq& 3a_q^2 \left(3n M^2+ \frac{n^2(\sigma^2 + \Delta^2)}{\tau^2} \right).
  \end{eqnarray*}
\EndProof
\end{proof}

\begin{lemma}
   Let ${\varphi}(z)$ is $\mu$-strongly-convex-strongly-concave (convex-concave with $\mu = 0$) and $\mathbf{e}$ be from $\mathcal{RS}^n_2(1)$. Then function $\hat{\varphi}(z)$ is $\mu$-strongly-convex-strongly-concave  and under Assumption 2 satisfies:
    \begin{eqnarray*}
        \sup_{z \in \mathcal{Z}} |\hat{\varphi}(z) - {\varphi}(z)| \leq \tau M. 
    \end{eqnarray*}
\end{lemma}
\begin{proof}
    Using definition \eqref{eq:smoothphi} of $\hat \varphi$:
    \begin{eqnarray*}
        \big|\hat{\varphi}(z) - {\varphi}(z)\big|= 
        \big|\mathbb{E}_{\mathbf{e}}[\varphi(z + \tau \mathbf{e})] - \varphi(z)\big| = \left|\mathbb{E}_{\mathbf{e}}\left[\varphi(z + \tau \mathbf{e}) - \varphi(z) \right] \right|.
    \end{eqnarray*}
    Since $\varphi(z)$ is $M$-Lipschitz, we get
    \begin{eqnarray*}
       \left|\mathbb{E}_{\mathbf{e}}\left[\varphi(z + \tau \mathbf{e}) - \varphi(z) \right]\right| &\leq&  \left|\mathbb{E}_{\mathbf{e}}\left[ M \|\tau \mathbf{e}\|_2 \right] \right| \leq M\tau .
    \end{eqnarray*}
    \EndProof
\end{proof}

\begin{lemma}
    Under Assumption 4 it holds that 
    \begin{eqnarray}
    \label{temp9858}
    \tilde \nabla \hat{\varphi}(z) = \mathbb{E}_{\mathbf{e}}\left[\frac{n\left(\varphi(z + \tau \mathbf{e}) -  \varphi(z - \tau \mathbf{e})\right)}{2\tau}\left(
    \begin{array}{c}
    \mathbf{e}_x\\
    -\mathbf{e}_y \\
    \end{array}\right)\right],\\
    \|\mathbb{E}_{\mathbf{e},\xi}[ g(z, \mathbf{e}, \tau, \xi^{\pm}) ] - \tilde \nabla \hat{\varphi}(z)\|_q \leq \frac{\Delta n a_q}{\tau}\nonumber\hspace{2cm}.
    \end{eqnarray}
\end{lemma}
\begin{proof}
    The proof of \eqref{temp9858} is given in \cite{Shamir15} and follows from the Stokes' theorem.
    Then 
    \begin{eqnarray*}
    \mathbb{E}_{\mathbf{e},\xi}[ g(z, \mathbf{e}, \tau, \xi^{\pm}) ] - \tilde \nabla \hat{\varphi}(z) &=&  \mathbb{E}_{\mathbf{e}} \left[ \frac{n \left(\delta(z + \tau \mathbf{e}) -  \delta(z - \tau \mathbf{e})\right)}{2\tau}\left(\begin{array}{c}
    \mathbf{e}_x\\
    -\mathbf{e}_y \\
    \end{array}
    \right)\right].
    \end{eqnarray*}
    Using inequalities \eqref{eq:PrSt2} and definition of $a_q$ completes the proof.
    \EndProof
\end{proof}

\textbf{Theorem 1.}
Let problem \eqref{problem} with function $\varphi(x,y)$ be solved using Algorithm \ref{alg1} with the oracle \eqref{eq:PrSt3}. Assume, that the set $\mathcal{Z}$, the convex-concave function $\varphi(x,y)$ and its inexact modification $\widetilde{\varphi}(x,y)$ satisfy Assumptions 1, 2, 4. Denote by $N$ the number of iterations and $\gamma_k = \gamma = const$. Then the rate of convergence is given by the following expression:
\begin{eqnarray*}
    \mathbb{E}\left[\varepsilon_{sad}(\bar z_{N})\right] &\leq&
    \frac{3\Omega^2}{2 \gamma (N+1)} + \frac{3\gamma M^2_{all}}{2} + \frac{\Delta \Omega n a_q}{\tau} + 2\tau M.
\end{eqnarray*} 
$\Omega$ is a diameter of $\mathcal{Z}$, $M^2_{all} = 3\left(3n M^2 +\frac{n^2(\sigma^2 + \Delta^2)}{\tau^2}\right)a^2_q$ and 
\begin{equation*}
    \varepsilon_{sad}(\bar z_{N}) = \max_{y' \in \mathcal{Y}} \varphi(\bar x_{N}, y') - \min_{x' \in \mathcal{X}} \varphi(x', \bar y_{N}).
\end{equation*}

\begin{proof} We divided the proof into three steps.

\textbf{Step 1.} Let $g_k \eqdef \gamma g(z_k, \mathbf{e}_k, \tau,  \xi^{\pm}_k)$. By the step of Algorithm \ref{alg1}, $z_{k + 1}= \text{prox}_{z_k}(g_k)$. Taking into account \eqref{lemma3_1}, we get that for all $u \in \mathcal{Z}$
\begin{eqnarray*}
    \langle g_k , z_{k+1} - u\rangle = \langle g_k , z_{k+1} - z_{k} + z_{k} - u\rangle \leq V_{z_k}(u) - V_{z_{k + 1}}(u) - V_{z_k}(z_{k + 1}).
\end{eqnarray*}
By simple transformations:
\begin{eqnarray*}
    \langle g_k , z_{k} - u\rangle &\leq& \langle g_k , z_{k} - z_{k+1}\rangle + V_{z_k}(u) - V_{z_{k + 1}}(u) - V_{z_k}(z_{k + 1}) \nonumber\\
    &\leq& \langle g_k , z_{k} - z_{k+1}\rangle + V_{z_k}(u) - V_{z_{k + 1}}(u) - \frac{1}{2}\|z_{k + 1} - z_{k}\|^2_p. 
\end{eqnarray*}
In last inequality we use the property of the Bregman divergence: $V_x(y) \ge \frac{1}{2}\|x-y\|_p^2$.
Using Hölder's inequality and the fact: $ab - \nicefrac{b^2}{2} \leqslant\nicefrac{a^2}{2}$, we have
\begin{eqnarray}
    \label{temp2}
    \langle g_k , z_{k} - u\rangle &\leq&
    \| g_k \|_q\|z_{k} - z_{k+1}\|_p  + V_{z_k}(u) - V_{z_{k + 1}}(u) - \frac{1}{2}\|z_{k + 1} - z_{k}\|^2_p  \nonumber\\
    &\leq&
    V_{z_k}(u) - V_{z_{k + 1}}(u) + \frac{1}{2}\| g_k \|^2_q.
\end{eqnarray}
Summing \eqref{temp2} over all $k$ from 0 to $N$ and by the definitions of $g_k$ and $\Omega$ (diameter of $\mathcal{Z}$): $\forall u \in \mathcal{Z}$
\begin{equation}
    \label{temp3}
    \gamma \sum^N_{k = 0}  \langle g(z_k, \mathbf{e}_k, \tau,  \xi^{\pm}_k) , z_{k} - u\rangle \leq\frac{\Omega^2}{2} + \frac{\gamma^2}{2}\sum^N_{k = 0}  \| g(z_k, \mathbf{e}_k, \tau,  \xi^{\pm}_k)\|^2_q.
\end{equation}
Let $\Delta_k \eqdef g(z_k, \mathbf{e}_k, \tau, \xi^{\pm}_k) -  \tilde\nabla \hat{\varphi}(z_k)$ and $D(u) \eqdef \sum^N_{k = 0} \gamma \langle \Delta_k  ,u - z_{k}\rangle$. Substituting the definition of $D(u)$ in \eqref{temp3}, we have for all $u \in \mathcal{Z}$
\begin{eqnarray}
    \label{temp_main}
    \gamma \sum^N_{k = 0}  \langle \tilde\nabla \hat{\varphi}(z_k) , z_{k} - u\rangle &\leq&\frac{\Omega^2}{2} + \frac{\gamma^2}{2}\sum^N_{k = 0}  \| g(z_k, \mathbf{e}_k, \tau,  \xi^{\pm}_k) \|^2_q +  D(u) .
\end{eqnarray}
By $\tilde \nabla \hat{\varphi}(z)$ we mean a block vector consisting of two vectors $\nabla_x \hat{\varphi}(x,y)$ and $-\nabla_y \hat{\varphi}(x,y)$.

\textbf{Step 2.} 
We consider a relationship between functions $\hat{\varphi}(z)$ and $\varphi(z)$. 
Combining \eqref{sad} and \eqref{lemma1_0} we get
\begin{eqnarray*}
    \varepsilon_{sad}(\bar z_N) 
    &\le& \max\limits_{y' \in \mathcal{Y}} \hat{\varphi}(\bar x_N, y') - \min\limits_{x' \in \mathcal{X}} \hat{\varphi}(x', \bar y_N) + 2\tau M. 
\end{eqnarray*}
Then, by the definition of $\bar x_N$ and $\bar y_N$ (see \eqref{sad}), Jensen's inequality and convexity-concavity of $\hat{\varphi}$:
\begin{eqnarray*}
    \varepsilon_{sad}(\bar z_N)
    &\leq& \max\limits_{y' \in \mathcal{Y}} \hat{\varphi}\left(\frac{1}{N+1} \left(\sum^N_{k = 0}  x_k\right) , y'\right) - \min\limits_{x' \in \mathcal{X}} \hat{\varphi}\left(x', \frac{1}{N+1} \left(\sum^N_{k = 0} y_k \right)\right) \nonumber \\
    &&+ 2\tau M
    \nonumber \\
    &\leq& \max\limits_{y' \in \mathcal{Y}} \frac{1}{N+1} \sum^N_{k = 0} \hat{\varphi}(x_k, y')  - \min\limits_{x' \in \mathcal{X}} \frac{1}{N+1} \sum^N_{k = 0} \hat{\varphi}(x', y_k) + 2\tau M.
\end{eqnarray*}
Given the fact of linear independence of $x'$ and $y'$:
\begin{eqnarray*}
    \varepsilon_{sad}(\bar z_N) &\leq& \max\limits_{(x', y') \in \mathcal{Z}}\frac{1}{N+1} \sum^N_{k = 0}  \left(\hat{\varphi}(x_k, y')  - \hat{\varphi}(x', y_k) \right) + 2\tau M.
\end{eqnarray*}
Using convexity and concavity of the function $\hat{\varphi}$:
\begin{eqnarray}
\label{temp8}
    \varepsilon_{sad}(\bar z_N) &\leq&  \max\limits_{(x', y') \in \mathcal{Z}}\frac{1}{N+1} \sum^N_{k = 1}  \left(\hat{\varphi}(x_k, y')  - \hat{\varphi}(x', y_k) \right) + 2\tau M  \nonumber \\
    &= & \max\limits_{(x', y') \in \mathcal{Z}} \frac{1}{N+1} \sum^N_{k = 1} \left(\hat{\varphi}(x_k, y') - \hat{\varphi}(x_k, y_k) + \hat{\varphi}(x_k, y_k) - \hat{\varphi}(x', y_k) \right) \nonumber \\
    &&+ 2\tau M  \nonumber \\
    &\leq& \max\limits_{(x', y') \in \mathcal{Z}} \frac{1}{N+1} \sum^N_{k = 1} \left(\langle \nabla_y \hat{\varphi} (x_k, y_k), y'-y_k \rangle + \langle \nabla_x \hat{\varphi} (x_k, y_k), x_k-x' \rangle \right) \nonumber \\
    &&+ 2\tau M \nonumber \\
    &\leq& \max\limits_{u \in \mathcal{Z}} \frac{1}{N+1} \sum^N_{k = 0}  \langle \tilde \nabla \hat{\varphi}(z_k), z_k - u\rangle  + 2\tau M.
\end{eqnarray}

\textbf{Step 3.} Combining expressions \eqref{temp_main}, \eqref{temp8}, \eqref{boundlem0} and taking full mathematical expectation, we get
\begin{eqnarray}
    \label{t4t}
    \mathbb{E}\left[\varepsilon_{sad}(\bar z_N)\right] 
    &\leq& \frac{\Omega^2}{2\gamma(N+1)} + \frac{\gamma M^2_{all}}{2} + \frac{1}{\gamma(N+1)}\mathbb{E}\left[\max_{u \in \mathcal{Z}} D(u)\right]  + 2\tau M .
\end{eqnarray}

Let's estimate $D(u)$. For this we prove the following lemma:
\begin{lemma}[see Lemma 5.3.2 from \cite{nemirovski}]\label{lemma3}
\begin{eqnarray}
    \label{19}
    \mathbb{E}\left[\max_{u \in \mathcal{Z}} D(u)\right] &\leq& 
    \Omega^2 + \frac{\gamma (N + 1)\Delta \Omega n a_q}{\tau} + \gamma^2 M^2_{all} (N + 1), 
\end{eqnarray}
where $M^2_{all} \eqdef
3\left(cn M^2 + \frac{n^2(\sigma^2 + \Delta^2)}{\tau^2}\right)a^2_q$ is from Lemma 1.
\end{lemma}
\begin{proof} Let define sequence $v$: $v_1 \eqdef z_1$, $v_{k+1} \eqdef \text{prox}_{v_k} (-\rho \gamma \Delta_k)$ for some $\rho > 0$:
\begin{eqnarray}
\label{temp177}
    D(u) &=& \gamma\sum\limits_{k=0}^N  \langle -\Delta_k, z_k - u \rangle \nonumber\\
    &=& \gamma\sum\limits_{k=0}^N  \langle -\Delta_k, z_k - v_k \rangle + \gamma
    \sum\limits_{k=0}^N  \langle -\Delta_k,  v_k - u \rangle . 
\end{eqnarray}
By the definition of $v$ and an optimal condition for the prox-operator, we have for all $u \in \mathcal{Z}$
\begin{eqnarray*}
    \langle -\gamma \rho\Delta_k - \nabla d(v_{k+1})  + \nabla d(v_{k+1}), u - v_{k+1} \rangle \geq 0.
\end{eqnarray*}
Rewriting this inequality, we get
\begin{eqnarray*}
    \langle -\gamma \rho\Delta_k, v_k  - u \rangle \leq \langle -\gamma \rho\Delta_k, v_k - v_{k+1} \rangle  + \langle \nabla d (v_{k+1}) - \nabla d (v_k), u - v_{k+1} \rangle.
\end{eqnarray*}
Using \eqref{temp111}:
\begin{eqnarray*}
    \langle -\gamma \rho\Delta_k, v_k  - u \rangle \leq \langle -\gamma \rho\Delta_k, v_k - v_{k+1} \rangle + V_{v_k}(u) -  V_{v_{k+1}}(u) - V_{v_k}(v_{k+1}).
\end{eqnarray*}
Bearing in mind the Bregman divergence property $2V_x(y) \geq \|x-y\|_p^2$:
\begin{eqnarray*}
    \langle -\gamma \rho\Delta_k, v_k  - u \rangle \leq \langle -\gamma \rho\Delta_k, v_k - v_{k+1} \rangle + V_{v_k}(u) -  V_{v_{k+1}}(u) - \frac{1}{2}\|v_{k+1} - v_k\|_p^2.
\end{eqnarray*}
Using the definition of the conjugate norm:
\begin{eqnarray*}
    \langle -\gamma \rho\Delta_k, v_k  - u \rangle &\leq& \|\gamma \rho\Delta_k\|_q\cdot\| v_k - v_{k+1} \|_p + V_{v_k}(u) -  V_{v_{k+1}}(u) - \frac{1}{2}\|v_{k+1} - v_k\|_p^2 \nonumber\\ 
    &\leq& \frac{\rho^2 \gamma^2}{2}\|\Delta_k\|_q^2 + V_{v_k}(u) -  V_{v_{k+1}}(u).
\end{eqnarray*}
Summing over $k$ from $0$ to $N$:
\begin{eqnarray*}
    \sum\limits_{k=0}^N \gamma \rho \langle -\Delta_k,  v_k - u \rangle \leq V_{v_1}(u) -  V_{v_{N+1}}(u) + \frac{\rho^2\gamma^2}{2}\sum\limits_{k=0}^N \|\Delta_k\|_q^2.
\end{eqnarray*}
Notice that $V_x(y) \geq 0$ and $V_{v_1}(u) \leq \nicefrac{\Omega^2}{2}$:
\begin{eqnarray}
\label{temp192}
    \sum\limits_{k=0}^N \gamma \langle -\Delta_k,  v_k - u \rangle \leq \frac{\Omega^2}{2\rho} + \frac{\rho \gamma^2}{2}\sum\limits_{k=0}^N \|\Delta_k\|_q^2.
\end{eqnarray}
Substituting \eqref{temp192} into \eqref{temp177}:
\begin{eqnarray*}
     D(u) &\leq& \sum\limits_{k=0}^N \gamma \langle \Delta_k, v_k - z_k  \rangle +
    \frac{\Omega^2}{2\rho} + \frac{\rho \gamma^2}{2}\sum\limits_{k=0}^N \|\Delta_k\|_q^2.  
\end{eqnarray*}
The right side is independent of $u$, then
\begin{eqnarray}
    \label{5t5}
     \max_{u \in \mathcal{Z}} D(u) &\leq& \sum\limits_{k=0}^N \gamma \langle \Delta_k, v_k - z_k  \rangle + 
    \frac{\Omega^2}{2\rho} + \frac{\rho \gamma^2}{2}\sum\limits_{k=0}^N \|\Delta_k\|_q^2.  
\end{eqnarray}
Taking the full expectation:
\begin{eqnarray*}
    \mathbb{E}\left[\max_{u \in \mathcal{Z}} D(u)\right] &\leq& \mathbb{E}\left[\sum\limits_{k=1}^N \gamma \langle \Delta_k, v_k - z_k \rangle\right] + 
    \frac{\Omega^2}{2\rho} + \frac{\rho \gamma^2}{2}\mathbb{E}\left[\sum\limits_{k=0}^N \|\Delta_k\|_q^2\right].  
\end{eqnarray*}
Using the independence of $\mathbf{e}_1, \ldots, \mathbf{e}_N, \xi^{\pm}_1, \ldots, \xi^{\pm}_N$, we have
\begin{eqnarray*}
    \mathbb{E}\left[\max_{u \in \mathcal{Z}} D(u)\right] &\leq& \mathbb{E}\left[\sum\limits_{k=0}^N \gamma \mathbb{E}_{\mathbf{e}_k, \xi_k}\left[\langle \Delta_k, v_k - z_k \rangle\right]\right]  + 
    \frac{\Omega^2}{2\rho} + \frac{\rho \gamma^2}{2}\mathbb{E}\left[\sum\limits_{k=0}^N \|\Delta_k\|_q^2\right]
    .  
\end{eqnarray*}
Note that $v_k - z_k$ does not depend on $\mathbf{e}_k$, $\xi_k$. Then 
\begin{eqnarray*}
    \mathbb{E}\left[\max_{u \in \mathcal{Z}} D(u)\right] &\leq& \mathbb{E}\left[\sum\limits_{k=0}^N \gamma \langle \mathbb{E}_{\mathbf{e}_k, \xi_k}\left[\Delta_k\right], v_k - z_k \rangle\right] + 
    \frac{\Omega^2}{2\rho} + \frac{\rho \gamma^2}{2}\mathbb{E}\left[\sum\limits_{k=0}^N \|\Delta_k\|_q^2\right].  
\end{eqnarray*}
By \eqref{temp986} and definition of diameter $\Omega$ we get
\begin{eqnarray*}
    \mathbb{E}\left[\max_{u \in \mathcal{Z}} D(u)\right] &\leq& \frac{\Delta \Omega n a_q}{\tau}\sum\limits_{k=0}^N \gamma + \frac{\Omega^2}{2\rho} + \frac{\rho \gamma^2}{2}\sum\limits_{k=0}^N \mathbb{E}\left[\|\Delta_k\|_q^2\right].  
\end{eqnarray*}
To prove the lemma, it remains to estimate $\mathbb{E}\left[\|\Delta_k\|_q^2\right]$:
\begin{eqnarray*}
    \mathbb{E}\left[\|\Delta_k\|_q^2\right] &\leq& \mathbb{E}\left[\|g(z_k, \mathbf{e}_k, \tau,  \xi^{\pm}_k, \delta^{\pm}_k) -  \tilde\nabla\hat{\varphi}(z_k)\|_q^2\right]  \nonumber\\
    &\leq& 2 \mathbb{E}\left[\|g(z_k, \mathbf{e}_k, \tau,  \xi^{\pm}_k, \delta^{\pm}_k)\|_q^2\right] +  2 \mathbb{E}\left[\|\tilde\nabla\hat{\varphi}(z_k)\|_q^2\right]\nonumber\\
    &\leq& 2 \mathbb{E}\left[\|g(z_k, \mathbf{e}_k, \tau,  \xi^{\pm}_k, \delta^{\pm}_k)\|_q^2\right] +  2 \mathbb{E}\left[\left\|\frac{n\left(\varphi(z + \tau \mathbf{e}) -  \varphi(z - \tau \mathbf{e})\right)}{2\tau}\mathbf{e}\right\|_q^2\right]
    .  
\end{eqnarray*}
Using Lemma \ref{beznos},  we have $\mathbb{E}\left[\|\Delta_k\|_q^2\right] \leq 4M^2_{all}$, whence
\begin{eqnarray*}
    \mathbb{E}\left[\max_{u \in \mathcal{Z}} D(u)\right] &\leq& 
    \frac{\Omega^2}{2\rho} + \frac{\gamma (N+1)\Delta \Omega n a_q}{\tau} + 2\rho\gamma^2 M^2_{all}(N+1).  
\end{eqnarray*}
Taking $\rho = \nicefrac{1}{2}$ ends the proof of lemma.
\EndProof
\end{proof}

\eqref{t4t} with this lemma gives
\begin{eqnarray*}
    \mathbb{E}\left[\varepsilon_{sad}(\bar z_N)\right] &\leq&
    \frac{3\Omega^2}{2 \gamma (N+1)} + \frac{3\gamma M^2_{all}}{2} + \frac{\Delta \Omega n a_q}{\tau} + 2\tau M.
\end{eqnarray*} 
This completes the proof of the theorem.
\EndProof
\end{proof}

\textbf{Theorem 2.} 
Let problem \eqref{problem} with function $\varphi(x,y)$ be solved using Algorithm \ref{alg1} with $V_z(w) = \frac{1}{2}\|z-w\|^2_2$ and the oracle \eqref{eq:PrSt3}. Assume, that the set $\mathcal{Z}$, the function  $\varphi(x,y)$ and its inexact modification $\widetilde{\varphi}(x,y)$ satisfy Assumptions 1, 2, 3, 4. Denote by $N$ the number of iterations and $\gamma_k = \frac{1}{\mu k}$. Then the rate of convergence is given by the following expression:
\begin{eqnarray*}
\EE\left[\varphi(\bar x_N, y^*) - \varphi(x^*, \bar y_N)\right]
    &\leq& \frac{M^2_{all} \log (N+1)}{2 \mu (N+1)} + \frac{\Delta n\Omega}{\tau} + 2\tau M
\end{eqnarray*}

\begin{proof}
We start this proof from substituting definition of $g_k$ and $u = z^*$ in \eqref{temp2}:
\begin{eqnarray*}
    2\gamma_k\langle g(z_k, \mathbf{e}_k, \tau,  \xi^{\pm}_k) , z_{k} - z^*\rangle &\leq&
    \|z_k - z^*\|^2 - \|z_{k+1} - z^*\|^2 + \gamma^2_k\| g(z_k, \mathbf{e}_k, \tau,  \xi^{\pm}_k) \|^2.
\end{eqnarray*}
With small rearrangement
\begin{eqnarray*}
    2\gamma_k \langle \tilde \nabla \hat\varphi(z_k) , z_{k} - z^*\rangle
    &\leq&
    \|z_k - z^*\|^2 - \|z_{k+1} - z^*\|^2 + \gamma^2_k\| g(z_k, \mathbf{e}_k, \tau,  \xi^{\pm}_k) \|^2 \\
    &&+ 2\gamma_k\langle \tilde \nabla \hat{\varphi}(z_k) - g(z_k, \mathbf{e}_k, \tau,  \xi^{\pm}_k) , z_{k} - z^*\rangle.
\end{eqnarray*}
On the other hand with \eqref{lemma1_0} and Lemma 3 we get
\begin{eqnarray*}
    \varphi(x_k, y^*) - \varphi(x^*, y_k) &=& \hat \varphi(x_k, y^*) + |\varphi(x_k, y^*) - \hat \varphi(x_k, y^*)| \\
    &&- \hat \varphi(x^*, y_k) + |\varphi(x^*, y_k) - \hat \varphi(x^*, y_k)|   \\
    &\leq & \hat \varphi(x_k, y^*) - \hat \varphi(x^*, y_k) + 2\tau M\\
    &\leq & \hat \varphi(x_k, y^*) - \hat \varphi(x_k, y_k) + \hat \varphi(x_k, y_k) - \hat \varphi(x^*, y_k) + 2\tau M\\
    &\leq & \la -\nabla_y \hat\varphi(x_k,y_k), y_k-y^*\ra - \frac{\mu}{2} \|y_k-y^* \|^2 \\
    &&+  \la -\nabla_x \hat\varphi(x_k,y_k), x_k-x^*\ra - \frac{\mu}{2} \|x_k-x^* \|^2 + 2\tau M \\
   & = & \la \widetilde{\nabla} \hat\varphi(z_k), z_k-z^*\ra - \frac{\mu}{2} \|z_k-z^* \|^2 + 2\tau M.
\end{eqnarray*}
By connecting we have
\begin{eqnarray*}
    2\gamma_k (\varphi(x_k, y^*) - \varphi(x^*, y_k))
    &\leq&
    (1- \mu\gamma_k)\|z_k - z^*\|^2 - \|z_{k+1} - z^*\|^2 + \gamma^2_k\| g(z_k, \mathbf{e}_k, \tau,  \xi^{\pm}_k) \|^2_2 \\
    &&+ 2\gamma_k\langle \tilde \nabla \hat{\varphi}(z_k) - g(z_k, \mathbf{e}_k, \tau,  \xi^{\pm}_k) , z_{k} - z^*\rangle + 4 \gamma_k\tau M.
\end{eqnarray*}
Taking the total expectation and taking into account that $z_{k} - z^*$ does not depend on $\mathbf{e}_k, \xi_k$:
\begin{eqnarray*}
    \EE[\varphi(x_k, y^*) - \varphi(x^*, y_k)]
    &\leq&
    \left(\frac{1}{2\gamma_k}- \frac{\mu}{2}\right)\EE\|z_k - z^*\|^2  \\
    && - \frac{1}{2\gamma_k}\EE\|z_{k+1} - z^*\|^2 + \frac{\gamma_k}{2}\EE\| g(z_k, \mathbf{e}_k, \tau,  \xi^{\pm}_k) \|^2_2 \\
    &&+ \EE\langle \EE_{\mathbf{e}_k, \xi_k}[\tilde \nabla \hat{\varphi}(z_k) - g(z_k, \mathbf{e}_k, \tau,  \xi^{\pm}_k)] , z_{k} - z^*\rangle + 2\tau M.
\end{eqnarray*}
With \eqref{boundlem0}, \eqref{temp986} with $a_q = 1$ (Euclidean case) we get
\begin{eqnarray*}
    \EE[\varphi(x_k, y^*) - \varphi(x^*, y_k)]
    &\leq&
    \left(\frac{1}{2\gamma_k}- \frac{\mu}{2}\right)\EE\|z_k - z^*\|^2 - \frac{1}{2\gamma_k}\EE\|z_{k+1} - z^*\|^2 \\
    &&+ \frac{\gamma_k M^2_{all}}{2} + \frac{\Delta n\Omega}{\tau} + 2\tau M.
\end{eqnarray*}
Summing over all $k$ from $0$ to $N$, we have
\begin{eqnarray*}
    \EE\left[\sum\limits_{k=0}^N\varphi(x_k, y^*) - \sum\limits_{k=0}^N\varphi(x^*, y_k)\right]
    &\leq&
    \sum\limits_{k=1}^{N-1} \left(\frac{1}{2\gamma_k} - \frac{1}{2\gamma_{k-1}} - \frac{\mu}{2}\right)\EE\|z_k - z^*\|^2\\
    &&+ \left(\frac{1}{2\gamma_0} - \frac{\mu}{2}\right) \|z_0 - z^*\|^2  + \frac{M^2_{all}}{2} \sum\limits_{k=0}^N \gamma_k \\
    &&+ \frac{\Delta n\Omega (N+1)}{\tau} + 2\tau M(N+1).
\end{eqnarray*}
With $\gamma_k = \frac{1}{\mu (k+1)}$ we get
\begin{eqnarray*}
    \EE\left[\sum\limits_{k=0}^N\varphi(x_k, y^*) - \sum\limits_{k=0}^N\varphi(x^*, y_k)\right]
    &\leq& \frac{M^2_{all} \log (N+1)}{2 \mu} + \frac{\Delta n\Omega (N+1)}{\tau} + 2\tau M(N+1).
\end{eqnarray*}
It remains only to apply Jensen's inequality to the left-hand side:
\begin{eqnarray*}
    \EE\left[\varphi(\bar x_N, y^*) - \varphi(x^*, \bar y_N)\right]
    &\leq& \frac{M^2_{all} \log (N+1)}{2 \mu(N+1)} + \frac{\Delta n\Omega}{\tau} + 2\tau M.
\end{eqnarray*}

\EndProof
\end{proof}

\begin{lemma}
    For $\tilde g_k \eqdef \tilde g(z_k, z_{k-1},  \mathbf{e}_k, \mathbf{e}_{k-1},  \xi_k, \xi_{k-1})$ defined in \eqref{eq:PrSt5} under Assumptions 2 and 4 the following inequalities holds:
    \begin{eqnarray*}
    \mathbb{E}\left[\|\tilde g_k\|^2_2\right] 
    &\leq& \alpha^k\mathbb{E}\left[\|\tilde g_{0}\|^2_2 \right] + \left(\frac{12n^2 (\sigma^2 + \Delta^2)}{\tau^2}+ 12n^2M^2\right) \frac{1}{1 - \alpha},
    \end{eqnarray*}
     where  $\alpha = \frac{6\gamma^2 n^2M^2}{\tau^2} < 1$.
\end{lemma}

\begin{proof}
\begin{eqnarray*}
    \mathbb{E}\left[\|\tilde g(z_k, z_{k-1},  \mathbf{e}_k, \mathbf{e}_{k-1},  \xi_k, \xi_{k-1})\|^2_2\right] \hspace{9cm}\nonumber\\
    = \frac{n^2}{\tau^2}\mathbb{E}\left[\left(\tilde \varphi(z_k + \tau \mathbf{e}_k,  \xi_k) -  \tilde \varphi(z_{k-1} + \tau \mathbf{e}_{k-1},  \xi_{k-1})\right)^2\right] \hspace{5.75cm}\nonumber\\
    = \frac{n^2}{\tau^2}\mathbb{E}\left[\left(\varphi(z_k + \tau \mathbf{e}_k) + \xi_k + \delta(z_k + \tau \mathbf{e}_k) -  \varphi(z_{k-1} + \tau \mathbf{e}_{k-1}) -  \xi_{k-1} - \delta(z_{k-1} + \tau \mathbf{e}_{k-1})\right)^2\right]. 
\end{eqnarray*}
With a simple fact \eqref{eq:squared_sum}, we get
\begin{eqnarray*}
    \mathbb{E}\left[\|\tilde g_k\|^2_2\right] &\leq& \frac{6n^2}{\tau^2}\mathbb{E}\left[\xi^2_k + \delta^2(z_k + \tau \mathbf{e}_k) + \xi^2_{k-1} + \delta^2(z_{k-1} + \tau \mathbf{e}_{k-1})\right]\nonumber\\
    &&+ \frac{6n^2}{\tau^2}\mathbb{E}\left[(\varphi(z_k + \tau \mathbf{e}_k) - \varphi(z_{k-1} + \tau \mathbf{e}_k))^2 \right]\nonumber\\
    &&+ \frac{6n^2}{\tau^2}\mathbb{E}\left[(\varphi(z_{k-1} + \tau \mathbf{e}_{k-1}) - \varphi(z_{k-1} + \tau \mathbf{e}_k))^2 \right]. 
\end{eqnarray*}
Next we use \eqref{bound} and \eqref{eq:PrSt2} and have
\begin{eqnarray*}
    \mathbb{E}\left[\|\tilde g_k\|^2_2\right] &\leq& \frac{12n^2 (\sigma^2 + \Delta^2)}{\tau^2}+ \frac{6n^2M^2}{\tau^2}\mathbb{E}\left\|z_k - z_{k-1}\|^2_2 \right]+ 6n^2M^2\mathbb{E}\left[\|\mathbf{e}_{k-1} - \mathbf{e}_k\|^2_2 \right] \nonumber\\
    &\leq& \frac{12n^2 (\sigma^2 + \Delta^2)}{\tau^2}+ \frac{6n^2M^2}{\tau^2}\mathbb{E}\left\|z_k - z_{k-1}\|^2_2 \right]+ 12n^2M^2. 
\end{eqnarray*}
Considering the step of Algorithm 1 we can rewrite as follows:
\begin{eqnarray*}
    \mathbb{E}\left[\|\tilde g_k\|^2_2\right] 
    &\leq& \frac{12n^2 (\sigma^2 + \Delta^2)}{\tau^2}+ \frac{6\gamma^2 n^2M^2}{\tau^2}\mathbb{E}\left\|\tilde g_{k-1}\|^2_2 \right]+ 12n^2M^2. 
\end{eqnarray*}
Then we run recursion
\begin{eqnarray*}
    \mathbb{E}\left[\|\tilde g_k\|^2_2\right] 
    &\leq& \left(\frac{6\gamma^2 n^2M^2}{\tau^2}\right)^k\mathbb{E}\left\|\tilde g_{0}\|^2_2 \right] + \left(\frac{12n^2 (\sigma^2 + \Delta^2)}{\tau^2}+ 12n^2M^2\right) \sum\limits_{i=0}^{k-1} \left(\frac{6\gamma^2 n^2M^2}{\tau^2}\right)^i. 
\end{eqnarray*}
With $\alpha = \frac{6\gamma^2 n^2M^2}{\tau^2} < 1$
\begin{eqnarray*}
    \mathbb{E}\left[\|\tilde g_k\|^2_2\right] 
    &\leq& \alpha^k\mathbb{E}\left\|\tilde g_{0}\|^2_2 \right] + \left(\frac{12n^2 (\sigma^2 + \Delta^2)}{\tau^2}+ 12n^2M^2\right) \frac{1}{1 - \alpha}. 
\end{eqnarray*}
\EndProof
\end{proof}

\textbf{Theorem 3.}
Let problem \eqref{problem} with function $\varphi(x,y)$ be solved using Algorithm \ref{alg1} with $V_z(w) = \frac{1}{2}\|z-w\|^2_2$ and the oracle \eqref{eq:PrSt5}. Assume, that the set $\mathcal{Z}$, the convex-concave function $\varphi(x,y)$ and its inexact modification $\widetilde{\varphi}(x,y)$ satisfy Assumptions 1, 2, 4. Denote by $N$ the number of iterations and $\gamma_k = \gamma= const$. Then the rate of convergence is given by the following expression:
\begin{eqnarray*}
    \mathbb{E}\left[\varepsilon_{sad}(\bar z_{N})\right] &\leq&
    \frac{3\Omega^2}{2\gamma(N+1)} + \frac{3\gamma}{2(N+1)(1 - \alpha)} \mathbb{E}\left[\|\tilde g_{0}\|^2_2 \right] \nonumber\\
    && + \frac{3\gamma}{2(1 - \alpha)} \left(\frac{12n^2 (\sigma^2 + \Delta^2)}{\tau^2}+ 12n^2M^2\right)  + 2\tau M  + \frac{\Delta \Omega n}{\tau}.
\end{eqnarray*} 
$\Omega$ is a diameter of $\mathcal{Z}$, $\alpha = \frac{6\gamma^2 n^2M^2}{\tau^2} < 1$.

\begin{proof} We begin our proof right away by obtaining the inequality similarly to \eqref{t4t} but by \eqref{boundlem1}, not \eqref{boundlem0}
\begin{eqnarray}
    \label{t4t1}
    \mathbb{E}\left[\varepsilon_{sad}(\bar z_N)\right] 
    &\leq& \frac{\Omega^2}{2\gamma(N+1)} + \frac{\gamma}{2(N+1)} \mathbb{E}\left[\|\tilde g_{0}\|^2_2 \right] \sum\limits_{k=0}^N \alpha^k  \nonumber\\
    && + \frac{\gamma}{2(1 - \alpha)} \left(\frac{12n^2 (\sigma^2 + \Delta^2)}{\tau^2}+ 12n^2M^2\right)  \nonumber\\
    &&+ \frac{1}{\gamma(N+1)}\mathbb{E}\left[\max_{u \in \mathcal{Z}} \tilde D(u)\right]  + 2\tau M  \nonumber\\
    &\leq& \frac{\Omega^2}{2\gamma(N+1)} + \frac{\gamma}{2(N+1)(1 - \alpha)} \mathbb{E}\left[\|\tilde g_{0}\|^2_2 \right] \nonumber\\
    && + \frac{\gamma}{2(1 - \alpha)} \left(\frac{12n^2 (\sigma^2 + \Delta^2)}{\tau^2}+ 12n^2M^2\right)  \nonumber\\
    &&+ \frac{1}{\gamma(N+1)}\mathbb{E}\left[\max_{u \in \mathcal{Z}} \tilde D(u)\right]  + 2\tau M, 
\end{eqnarray}
where $\tilde D(u) \eqdef \sum^N_{k = 0} \gamma \langle \tilde \Delta_k  ,u - z_{k}\rangle$ with $\tilde \Delta_k \eqdef \tilde g_k -  \tilde\nabla \tilde{\varphi}(z_k)$. Let estimate $\tilde D(u)$. For this we prove the following lemma:
\begin{lemma}
\begin{eqnarray}
    \label{191}
    \mathbb{E}\left[\max_{u \in \mathcal{Z}} \tilde D(u)\right] &\leq& \frac{\gamma(N+1)\Delta \Omega n}{\tau} + \Omega^2 + \frac{\gamma^2}{1 - \alpha} \|\tilde g_{0}\|^2_2 \nonumber\\
    &&+ \frac{\gamma^2 (N+1)}{1 - \alpha} \left(\frac{12n^2 (\sigma^2 + \Delta^2)}{\tau^2}+ 12n^2M^2\right). 
\end{eqnarray}
\end{lemma}
\begin{proof}
Let's start with \eqref{5t5}. All other steps are done in the same way.

\begin{eqnarray*}
     \max_{u \in \mathcal{Z}} \tilde D(u) &\leq& \sum\limits_{k=0}^N \gamma \langle \tilde\Delta_k, v_k - z_k  \rangle + 
    \frac{\Omega^2}{2\rho} + \frac{\rho \gamma^2}{2}\sum\limits_{k=0}^N \|\tilde\Delta_k\|_2^2.  
\end{eqnarray*}
Taking the full expectation: 
\begin{eqnarray*}
    \mathbb{E}\left[\max_{u \in \mathcal{Z}} \tilde D(u)\right] &\leq& \mathbb{E}\left[\sum\limits_{k=1}^N \gamma \langle \tilde\Delta_k, v_k - z_k \rangle\right] + 
    \frac{\Omega^2}{2\rho} + \frac{\rho \gamma^2}{2}\mathbb{E}\left[\sum\limits_{k=0}^N \|\tilde \Delta_k\|_2^2\right].  
\end{eqnarray*}
Using the independence of $\mathbf{e}_1, \ldots, \mathbf{e}_N, \xi^{\pm}_1, \ldots, \xi^{\pm}_N$, we have
\begin{eqnarray*}
    \mathbb{E}\left[\max_{u \in \mathcal{Z}} D(u)\right] &\leq& \mathbb{E}\left[\sum\limits_{k=0}^N \gamma \mathbb{E}_{\xi_k}\left[\langle \tilde\Delta_k, v_k - z_k \rangle\right]\right] + \mathbb{E}\left[\sum\limits_{k=0}^N \gamma \mathbb{E}_{\mathbf{e}_k}\left[\langle \Delta_k - \tilde\Delta_k, v_k - z_k \rangle\right]\right] \nonumber\\ 
       &&+ 
    \frac{\Omega^2}{2\rho} + \frac{\rho \gamma^2}{2}\mathbb{E}\left[\sum\limits_{k=0}^N \|\Delta_k\|_2^2\right]
    .  
\end{eqnarray*}
Note that $v_k - z_k$ does not depend on $\mathbf{e}_k$, $\xi_k$ and $\mathbb{E}_{\xi_k} \tilde\Delta_k = 0$. Then 
\begin{eqnarray*}
    \mathbb{E}\left[\max_{u \in \mathcal{Z}} \tilde D(u)\right] &\leq& \mathbb{E}\left[\sum\limits_{k=0}^N \gamma \langle \mathbb{E}_{\mathbf{e}_k}\left[ \tilde\Delta_k\right], v_k - z_k \rangle\right] +\frac{\Omega^2}{2\rho} + \frac{\rho \gamma^2}{2}\mathbb{E}\left[\sum\limits_{k=0}^N \|\tilde \Delta_k\|_2^2\right]  \nonumber\\ 
    &=& \mathbb{E}\left[\sum\limits_{k=0}^N \gamma \langle \mathbb{E}_{\mathbf{e}_k}\left[\tilde\Delta_k\right], v_k - z_k \rangle\right] + 
    \frac{\Omega^2}{2\rho} + \frac{\rho \gamma^2}{2}\mathbb{E}\left[\sum\limits_{k=0}^N \|\tilde \Delta_k\|_2^2\right].  
\end{eqnarray*}
By \eqref{temp986} and definition of diameter $\Omega$ we get
\begin{eqnarray*}
    \mathbb{E}\left[\max_{u \in \mathcal{Z}} \tilde D(u)\right] &\leq& \frac{\Delta \Omega n}{\tau}\sum\limits_{k=0}^N \gamma + \frac{\Omega^2}{2\rho} + \frac{\rho \gamma^2}{2}\sum\limits_{k=0}^N \mathbb{E}\left[\|\tilde \Delta_k\|_2^2\right].  
\end{eqnarray*}
To prove the lemma, it remains to estimate $\mathbb{E}\left[\|\tilde \Delta_k\|_2^2\right]$:
\begin{eqnarray*}
    \mathbb{E}\left[\|\tilde \Delta_k\|_2^2\right] &\leq&  2 \mathbb{E}\left[\| \tilde g_k\|_2^2\right] +  2 \mathbb{E}\left[\left\|\frac{n\left(\varphi(z + \tau \mathbf{e}) -  \varphi(z - \tau \mathbf{e})\right)}{2\tau}\mathbf{e}\right\|_2^2\right]
    .  
\end{eqnarray*}
Using Lemma \ref{beznos1},  we have 
\begin{eqnarray*}
    \mathbb{E}\left[\max_{u \in \mathcal{Z}} \tilde D(u)\right] &\leq& \frac{\Delta \Omega n}{\tau}\sum\limits_{k=0}^N \gamma + \frac{\Omega^2}{2\rho} + 2\rho \gamma^2\sum\limits_{k=0}^N \alpha^k\mathbb{E}\left[\|\tilde g_{0}\|^2_2 \right] \nonumber\\
    &&+ 2\rho \gamma^2 (N+1)\left(\frac{12n^2 (\sigma^2 + \Delta^2)}{\tau^2}+ 12n^2M^2\right) \frac{1}{1 - \alpha}.  
\end{eqnarray*}
Taking $\rho = \nicefrac{1}{2}$ ends the proof of lemma.
\EndProof
\end{proof}
\eqref{t4t1} with this lemma gives
\begin{eqnarray*}
    \mathbb{E}\left[\varepsilon_{sad}(\bar z_N)\right] &\leq&
     \frac{3\Omega^2}{2\gamma(N+1)} + \frac{3\gamma}{2(N+1)(1 - \alpha)} \mathbb{E}\left[\|\tilde g_{0}\|^2_2 \right] \nonumber\\
    && + \frac{3\gamma}{2(1 - \alpha)} \left(\frac{12n^2 (\sigma^2 + \Delta^2)}{\tau^2}+ 12n^2M^2\right)  + 2\tau M  + \frac{\Delta \Omega n}{\tau}.
\end{eqnarray*} 
This completes the proof of the theorem.
\EndProof
\end{proof}

\section{Proofs for Section \ref{32}}
The proofs of the Theorems \ref{th_smooth} and \ref{th_smooth_strong} copy the proofs of the Theorems \ref{th_main} and \ref{th_main1} except for the usage Lemma \ref{lemma_ltau^2} instead of Lemma \ref{lemma1}. So the term $2M\tau$ in Theorems \ref{th_main} and \ref{th_main1} is  replaced by the term $L\tau^2$ in Theorems \ref{th_smooth} and \ref{th_smooth_strong}.

\section{Proofs for Section \ref{33}}

\textbf{Theorem 6.}
Let $\varphi \in {\cal F}_{\mu, \beta}(L)$ with $\mu$, $L > 0$ and $\beta > 2$. 
Let Assumption \ref{noise-ass} hold 
and let $\mathcal{Z}$ be a convex compact subset of $\R^n$.
Let $\varphi$ be $M$-Lipschitz on the Euclidean $\tau_1$-neighborhood of $\mathcal{Z}$ (see $\tau_k$ below). 

Then the rate of convergence is given by Algorithm \ref{algo1} with parameters
\begin{equation*}
    \tau_k = \left(\dfrac{3\kappa\sigma^2n}{2(\beta-1)(\kappa_\beta L)^2}\right)^{\frac{1}{2\beta}}k^{-\frac{1}{2\beta}}, \quad \alpha_k=\dfrac{2}{\mu k}, \quad k =1,\dots, N
\end{equation*}
satisfies 
\begin{equation*}
\begin{split}
    \E \left[ \varphi(\overline{x}_N, y^*) - \varphi(x^*, \overline{y}_N) \right] &\leq \max_{y\in \mathcal{Y}} \E \left[\varphi(\overline{x}_N, y)  \right] -
    \min_{x\in \mathcal{X}}\E \left[\varphi(x, \overline{y}_N) \right] \\
    &\leq \dfrac{1}{\mu} \left(n^{2-\frac{1}{\beta}}\dfrac{A_1}{N^{\frac{\beta-1}{\beta}}}+A_2\dfrac{n(1+\ln N)}{N} \right),
\end{split}
\end{equation*}
where $\overline{z}_N = \frac{1}{N}\sum\limits_{k=1}^N z_k$, $A_1=3\beta(\kappa \sigma^2)^{\frac{\beta-1}{\beta}}(\kappa_{\beta}L)^{\frac{2}{\beta}}$, $A_2=9\kappa G^2$, $\kappa_{\beta}$ and $\kappa$ are constants depending only on $\beta$, see \eqref{kappa-beta} and \eqref{kappa-squared}.

\begin{proof}
{\bf Step 1.}
Fix an arbitrary $z \in \mathcal{Z}$. As $z_{k+1}$ is the Euclidean projection we have $\|z_{k+1} - z\|^2 \leq \| z_k - \gamma_k \widetilde{g_k} - z \|^2 $ which is equivalent to 
\begin{equation}
    \langle \widetilde{g_k}, z_k-z \rangle \leq \dfrac{\| z_k - z \|^2 - \| z_{k+1} - z \|^2}{2\gamma_k} + \dfrac{\gamma_k}{2} \| \widetilde{g_k} \|^2.
\end{equation}

Using the strong convexity-concavity and combining $x$ and $y$ parts of the argument $z$ together  we have 
\begin{equation}
\begin{split}
    \varphi(x_k, y) - \varphi(x, y_k) 
    =& \varphi(x_k, y) - \varphi(x_k, y_k) + \varphi(x_k, y_k) - \varphi(x, y_k)  \\
    \leq & \la -\nabla_y \varphi(x_k,y_k), y_k-y\ra - \frac{\mu}{2} \|y_k-y \|^2 \\
    + & \la -\nabla_x \varphi(x_k,y_k), x_k-x\ra - \frac{\mu}{2} \|x_k-x \|^2 \\
    = & \la \widetilde{\nabla} \varphi(z_k), z_k-z\ra - \frac{\mu}{2} \|z_k-z \|^2.
\end{split}
\end{equation}

Combining the last two inequations we obtain
\begin{equation}
\begin{split}
\varphi(x_k, y) - \varphi(x, y_k) 
\leq & \langle \widetilde{\nabla} \varphi(z_k) - \widetilde{g_k}, z_k - z \rangle 
+ \dfrac{\| z_k - z \|^2 - \| z_{k+1} - z \|^2}{2\gamma_k} \\
 & + \dfrac{\gamma_k}{2} \| \widetilde{g_k} \|^2
    - \dfrac{\mu}{2} \| z_k - z \|^2. \\
\end{split}
\end{equation}

Taking conditional expectation given $z_k$ with respect to $r_{k}$, $\xi^+_{k}$ and $\xi^-_{k}$ we obtain 
\begin{equation}\label{regret_bound}
\begin{split}
\varphi(x_k, y) - \varphi(x, y_k) \leq & \langle \widetilde{\nabla} \varphi(z_k) -  \E \left[\widetilde{g_k} | z_k \right] , z_k - z \rangle  + \dfrac{\gamma_k}{2} \E \left[ \| \widetilde{g_k} \|^2 | z_k \right] \\
& +  \dfrac{\| z_k - z \|^2 - \E \left[\| z_{k+1} - z \|^2 | z_k \right]}{2\gamma_k} - \dfrac{\mu}{2} \| z_k - z \|^2. \\
\end{split}
\end{equation}

{\bf Step 2 (Bounding bias term).} Our aim is to bound the first term in \eqref{regret_bound}, namely $\langle \widetilde{\nabla} \varphi(z_k) -  \E \left[\widetilde{g_k} | z_k \right] , z_k - z \rangle $.
Using the Taylor expansion we have
\begin{equation}
\begin{split}
    \varphi \left(z_k + \tau_k r_k e_k\right)
    =& \varphi (z_k) + \langle \nabla \varphi(z_k), \tau_k r_k e_k \rangle \\
    +& \sum_{2\leq |m| \leq l} \dfrac{(\tau_k r_k)^{|m|}}{m!} D^{(m)}\varphi (z_k) e_k^{m} + R(\tau_k r_k e_k),
\end{split}
\end{equation}
where by assumption $|R(\tau_k r_k e_k)| \leq L \|\tau_k r_k e_k\|^{\beta} = L (\tau_k \cdot |r_k|)^{\beta} $.
Thus, 
\begin{equation}\label{grad_taylor_expansion}
\begin{split}
    \widetilde{g_k} = &
     \Bigl( \langle \nabla \varphi(x_k), \tau_k r_k e_k \rangle 
    + \sum_{2\leq |m| \leq l, |m| \text{ odd}} \dfrac{(\tau_k r_k)^{|m|}}{m!} D^{(m)}\varphi(z_k) e_k^{m}\\
    &+ \frac{1}{2} R(\tau_k r_k e_k) - \frac{1}{2} R(-\tau_k r_k e_k) + \xi^+_k - \xi^-_k \Bigl) \dfrac{n}{\tau_k} K(r_k) \left(
    \begin{array}{c}
    (\mathbf{e}_k)_x\\
    -(\mathbf{e}_k)_y \\
    \end{array}\right) .
\end{split}
\end{equation}

Using the properties of the smoothing kernel $K$, independence of $\mathbf{e}_k$ and $r_k$ (Assumption \ref{noise-ass}) and the fact that $\E\left[ e_k e_k^{T} \right] = \frac 1 n \mathbb{I}_{n \times n}$ we obtain
\begin{equation}\label{first_order_term}
    \E_{e_k, r_k} \left[ \left\langle \nabla \varphi(z_k), \tau_k r_k e_k \right\rangle \dfrac{n}{\tau_k} K(r_k) \left(
    \begin{array}{c}
    (\mathbf{e}_k)_x\\
    -(\mathbf{e}_k)_y \\
    \end{array}\right)
    \middle| z_k \right] = \widetilde{\nabla} \varphi(z_k).
\end{equation}

Using the fact that $\E \left[ r_k^{|m|} K(r_k) \right] = 0$ if $2 \leq |m| \leq l$ or $|m| = 0$ and Assumption \ref{noise-ass} we have
\begin{equation}\label{mid_order_term}
    \E\left[\Bigl(\sum_{2\leq |m| \leq l, |m| \text{ odd}} \dfrac{(\tau_k r_k)^{|m|}}{m!} D^{(m)}\varphi(z_k) e_k^{m}
    + \xi^+_k - \xi^-_k \Bigl) \dfrac{n}{\tau_k} K(r_k) \left(
    \begin{array}{c}
    (\mathbf{e}_k)_x\\
    -(\mathbf{e}_k)_y \\
    \end{array}\right) \middle|x_k \right]= 0.
\end{equation}

Substituting \eqref{grad_taylor_expansion}, \eqref{first_order_term} and \eqref{mid_order_term} in the first term in \eqref{regret_bound} and using the definition of $\kappa_{\beta}$ (see \eqref{kappa-beta}) we obtain
\begin{multline}\label{step2_prefinal}
    \left|   \langle \widetilde{\nabla} \varphi(z_k) -  \E \left[\widetilde{g_k} | z_k \right] , z_k - z \rangle \right|
     = \\
     =  \left| \E\left[ \left( \frac{1}{2} R(\tau_k r_k e_k) - \frac{1}{2} R(-\tau_k r_k e_k) \right) \dfrac{n}{\tau_k} K(r_k) \left\la \left(
    \begin{array}{c}
    (\mathbf{e}_k)_x\\
    -(\mathbf{e}_k)_y \\
    \end{array}\right), z_k - z \right\ra \middle| z_k \right] \right| \\
     \leq L\tau_k^{\beta - 1} \cdot \E_{r_k} \left[ |r_k|^{\beta} K(r_k) \right] \cdot n\left|\E_{e_k} \left[\left\la \mathbf{e}_k, z_k-z \right\ra \middle| z_k \right] \right| \\
     \leq 
     \kappa_{\beta} L \sqrt{n} \tau_k^{\beta - 1} \| z_k - z\|, \\
\end{multline}
where in the last two inequalities the symmetry of Euclidean sphere and the fact from concentration measure theory that $\left|\E_{e} \left[\la e, s \ra \right] \right|^2 \leq \E_{e} \left[\la e, s \ra^2 \right] = \frac{\|s\|^2}{n} $  were used . Applying the inequality $ab \leq \nicefrac{1}{2} (a^2 + b^2)$ to the last expression in \eqref{step2_prefinal}  we finally get 
\begin{equation}\label{step2}
    \left|   \langle \widetilde{\nabla} \varphi(z_k) -  \E \left[\widetilde{g_k} | z_k \right] , z_k - z \rangle \right|
    \leq  \dfrac{(\kappa_{\beta} L)^2}{\mu} n \tau_k^{2(\beta-1)} + \dfrac{\mu}{4} \|z_k-z\|^2 .
\end{equation}

{\bf Step 3 (Bounding second moment of gradient estimator). }
Our aim is to estimate $\E \left[ \| \widetilde{g_k} \|^2 | z_k \right]$ which is the second term in \eqref{regret_bound}. The expectation here is with respect to $r_k$, $\xi^+_k$ and $\xi^-_k$. To lighten the presentation and without loss of generality we drop the lower script $k$ in all quantities. 

We have
\begin{equation}
\begin{split}
    \| \widetilde{g} \|^2 =& \dfrac{n^2}{4\tau^2} \left\|(\varphi(z+\tau r e) - \varphi(z - \tau r e) + \xi^+ - \xi^-) K(r) \left(
    \begin{array}{c}
    \mathbf{e}_x\\
    -\mathbf{e}_y \\
    \end{array}\right)\right\|^2 \\
    =& \dfrac{n^2}{4\tau^2} \left((\varphi(z+\tau r e) - \varphi(z - \tau r e) + \xi^+ - \xi^-) \right)^2 K^2(r).\\
\end{split}
\end{equation}

Using the inequality $(a+b+c)^2 \leq 3(a^2 + b^2 + c^2)$ and Assumption \ref{noise-ass} we get
\begin{equation}\label{dispersion_bound}
    \E \left[ \| \widetilde{g} \|^2 | z \right] 
    \leq \dfrac{3n^2}{4\tau^2} \left( \E\left[(\varphi(z+\tau r e) - \varphi(z - \tau r e))^2 K^2(r) \middle| z \right] + 2\kappa \sigma^2 \right).
\end{equation}



Using the symmetry of Euclidean unit sphere and the inequality $(a+b)^2 \ \leq 2(a^2 + b^2)$ we obtain
\begin{multline}
    \E \left[ \left(\varphi(z+e) -\varphi(z-e)\right)^2 \middle| z \right] = \E_e \left[ \left(\varphi(z+e) -\varphi(z-e)\right)^2  \right] \\
    \leq \E_e \left[ \left( \left(\varphi(z+e) - \E_e [\varphi(z+e)]\right) - \left(\varphi(z-e) - \E_e [\varphi(z-e)]\right) \right)^2  \right]\\
    \leq 2\E_e \left[ \left(\varphi(z+e) - \E_e [\varphi(z+e)]\right)^2 \right] + 2\E_e \left[ \left(\varphi(z-e) - \E_e [\varphi(z-e)]\right)^2 \right] \\
    \leq 2\sqrt{\E_e \left[ \left(\varphi(z+e) - \E_e [\varphi(z+e)]\right)^4 \right]} + 2\sqrt{\E_e \left[ \left(\varphi(z-e) - \E_e [\varphi(z-e)]\right)^4 \right]} \\
    \leq \dfrac{12M^2}{n},
\end{multline}
where in the last inequality Lemma \ref{lem:lemma_9_shamir} was used, so we have
\begin{equation}\label{delta_f_bound}
    \E \left[ \left(\varphi(z+\tau re) -\varphi(z-\tau re)\right)^2 \middle| z \right] \leq \dfrac{12(\tau r)^2 M^2}{n} \leq \dfrac{12 \tau^2 M^2}{n}.
\end{equation}

By substituting \eqref{delta_f_bound} into \eqref{dispersion_bound}, using independence of $e$ and $r$ and returning the lower script $k$ we finally get
\begin{equation}\label{step3}
    \E \left[ \| \widetilde{g_k} \|^2 | z_k \right] 
    \leq \kappa \left(9 n M^2  + \dfrac{3(n\sigma)^2}{2\tau_k^2} \right).
\end{equation}

{\bf Step 4. } Let $\rho_k^2$ denote full expectation $\E[\|z_k-z\|^2]$. Substituting \eqref{step2} and \eqref{step3} into \eqref{regret_bound}, taking full expectation we obtain
\begin{equation}\label{step4_begin}
\begin{split}
     \E [\varphi(x_k, y) - \varphi(x, y_k) ]
    \leq& \dfrac{(\kappa_{\beta} L)^2}{\mu} n \tau_k^{2(\beta-1)} 
    + \dfrac{\gamma_k}{2} \kappa \left(9 n M^2  + \dfrac{3(n\sigma)^2}{2\tau_k^2} \right) \\
    & +  \dfrac{\rho_k^2 - \rho_{k+1}^2}{2\gamma_k} - \left(\dfrac{\mu}{2} - \dfrac{\mu}{4} \right) \rho_k^2 .\\
\end{split}
\end{equation}

Using the convexity-concavity of $\varphi$ and \eqref{step4_begin} we have
\begin{equation}\label{step4_mid}
\begin{split}
    \E \left[ \varphi\left(\overline{x}_N, y\right) - \varphi\left(x, \overline{y}_N\right) \right] 
    &\leq \dfrac{1}{N} \sum_{k=1}^N \varphi\left(x_k, y\right) - \dfrac{1}{N} \sum_{k=1}^N \varphi\left(x, y_k\right) \\
    & \leq \dfrac{1}{N}\sum\limits_{k=1}^N \left( \dfrac{(\kappa_{\beta} L)^2}{\mu} n \tau_k^{2(\beta-1)} 
    + \dfrac{\gamma_k}{2} \kappa \left(9 n M^2  + \dfrac{3(n\sigma)^2}{2\tau_k^2} \right) \right) \\
    & + \dfrac{1}{N} \sum\limits_{k=1}^N \left( \dfrac{\rho_k^2 - \rho_{k+1}^2}{2\gamma_k} - \dfrac{\mu}{4} \rho_k^2 \right) .\\
\end{split}
\end{equation}

Let $\rho_{N+1}^2 = 0$. Then setting $\gamma_k = \dfrac{2}{\gamma k}$ yields 
\begin{equation}\label{sum_rho}
\begin{split}
    \sum_{k=1}^{N} \left( \dfrac{\rho_k^2 - \rho_{k+1}^2}{2\gamma_k} - \dfrac{\mu}{4} \rho_k^2 \right) &\leq
    \rho_1^2 \left( \dfrac{1}{2\gamma_1} - \dfrac{\mu}{4} \right) + \sum_{k=2}^{N+1} \rho_k^2 \left( \dfrac{1}{2\gamma_k} - \dfrac {1}{2\gamma_{k-1}}- \dfrac{\mu}{4}\right) \\
    &=\rho_1^2 \left( \dfrac{\mu}{4}-\dfrac{\mu}{4} \right) + \sum_{k=2}^{N+1} \rho_k^2 \left( \dfrac{\mu}{4}-\dfrac{\mu}{4} \right) = 0.
\end{split}
\end{equation}

Substituting \eqref{sum_rho} into \eqref{step4_begin} with $\gamma_k = \frac{2}{\mu k}$ we obtain 
\begin{equation}\label{step4_prefinal}
\begin{split}
    \mathbb{E}  [ \varphi\left(\overline{x}_N , y\right)& - \varphi\left(x, \overline{y}_N\right) ]  \\
    &\leq \dfrac{1}{\mu N} \sum_{k=1}^N   \left((\kappa_{\beta} L)^2 n \tau_k^{2(\beta-1)} 
    + \kappa \left(9 n M^2  + \dfrac{3(n\sigma)^2}{2\tau_k^2} \right) \dfrac{1}{k} \right)\\
    &= \dfrac{1}{\mu N} \sum_{k=1}^N \left( \left[n \cdot (\kappa_{\beta} L)^2 \tau_k^{2(\beta-1)} +
     n^2 \cdot \dfrac{3\kappa \sigma^2}{2k\tau_k^2} \right] +  \dfrac{9 \kappa n M^2}{k}\right).\\
\end{split}
\end{equation}

If $\sigma > 0$ then $\tau_k = {\left(\dfrac{3\kappa \sigma^2 n}{2(\beta -1)(\kappa_{\beta} L)^2}\right)}^{\frac{1}{2\beta}} k^{-\frac{1}{2\beta}}$ is the minimizer of square brackets. Plugging this $\tau_k$ in \eqref{step4_prefinal} and using two inequalities: for the expression in square brackets $\sum\limits_{k=1}^N k^{-1+\nicefrac{1}{\beta}} \leq \beta N^{\nicefrac{1}{\beta}}$ (if $\beta > 2$) and for the term after square brackets $\sum\limits_{k=1}^{N} \frac{1}{k} \leq 1 + \ln N$ we get
\begin{equation*}
    \mathbb{E}  [ \varphi\left(\overline{x}_N , y\right) - \varphi\left(x, \overline{y}_N\right) ] \leq \dfrac{1}{\mu} \left(n^{2-\frac{1}{\beta}}\dfrac{A_1}{N^{\frac{\beta-1}{\beta}}}+A_2\dfrac{n(1+\ln{N})}{N} \right).
\end{equation*}
with $A_1$ and $A_2$ from the formulation of Theorem \ref{theorem_highorder_smooth_strongly_convex}. 

Taking the minimum over $x$ and the maximum over $y$ we finally obtain
\begin{equation*}
\begin{split}
    \E \left[ \varphi(\overline{x}_N, y^*) - \varphi(x^*, \overline{y}_N) \right] &\leq \max_{y\in \mathcal{Y}} \E \left[\varphi(\overline{x}_N, y)  \right] -
    \min_{x\in \mathcal{X}}\E \left[\varphi(x, \overline{y}_N) \right] \\
    &\leq \dfrac{1}{\mu} \left(n^{2-\frac{1}{\beta}}\dfrac{A_1}{N^{\frac{\beta-1}{\beta}}}+A_2\dfrac{n(1+\ln N)}{N} \right).
\end{split}
\end{equation*}
\EndProof
\end{proof}

\textbf{Theorem 7.}
Let $\varphi \in {\cal F}_{\beta}(L)$ with $L > 0$ and $\beta > 2$. 
Let Assumption \ref{noise-ass} hold 
and let $\cal{Z}$ be a convex compact subset of $\R^n$.
Let $\varphi$ be $M$-Lipschitz on the Euclidean $\tau_1$-neighborhood of $\cal{Z}$ ($\tau_k$ is parameter from Theorem \ref{theorem_highorder_smooth_strongly_convex} for the regularized function $\varphi_\mu(z)$ whose  description is given below). Let $\overline{z}_N$ denote $\frac{1}{N}\sum\limits_{k=1}^N z_k$.

Let's define $N(\varepsilon)$:

\begin{equation*}
    N(\varepsilon)=\max\left\{ \left(R\sqrt{2A_1}\right)^{\frac{2\beta}{\beta-1}}\dfrac{n^{2+\frac{1}{\beta-1}}}{\varepsilon^{2+\frac{2}{\beta-1}}},\left(R\sqrt{2c'A_2}\right)^{2(1+\rho)}\dfrac{n^{1+\rho}}{\varepsilon^{2(1+\rho)}}\right\},
\end{equation*}
where  $A_1=3\beta(\kappa \sigma^2)^{\frac{\beta-1}{\beta}}(\kappa_{\beta}L)^{\frac{2}{\beta}}$, $A_2=9\kappa G^2$ -- constants from Theorem \ref{theorem_highorder_smooth_strongly_convex}, $\rho > 0$ -- arbitrarily small positive number, $c'$ -- constant which depends on $\rho$. 

Then the rate of convergence is given by the following expression:
\begin{equation}\label{error_less_than_eps}
    \E \left[ \varphi(\overline{x}_N, y^*) - \varphi(x^*, \overline{y}_N) \right] \leq \max_{y\in \mathcal{Y}} \E \left[\varphi(\overline{x}_N, y)  \right] -
    \min_{x\in \mathcal{X}}\E \left[\varphi(x, \overline{y}_N) \right] \leq \varepsilon
\end{equation}
after $N(\varepsilon)$ steps of Algorithm \ref{algo1} with settings from Theorem \ref{theorem_highorder_smooth_strongly_convex} for the regularized function: $\varphi_{\mu}(z):=\varphi(z)+\frac{\mu}{2} \| x-x_0 \|^2-\frac{\mu}{2} \| y - y_0 \|^2$, where $\mu \leq \frac{\varepsilon}{R^2}$, $R=\|z_0-z^*\|$, $z_0 \in \mathcal{Z}$ -- arbitrary point.

\begin{proof}
{\bf Step 1.} 
Let $z^*=(x^*, y^*)$ and $z_{\mu}^* = (x_\mu^*, y_\mu^*)$ denote the solutions of the saddle-point problems for functions $\varphi(z)$ and $\varphi_{\mu}(z)$ respectively. 
Setting $\mu = \frac{\varepsilon}{R^2}$ and using the inequality $\varphi_\mu\left(\overline{x}_N , y^*\right) - \varphi_\mu\left(x^*, \overline{y}_N\right) \leq \varphi_\mu\left(\overline{x}_N , y_\mu^*\right) - \varphi_\mu\left(x_\mu^*, \overline{y}_N\right)$ we obtain

\begin{equation}\label{connection}
    \begin{split}
         \,& \E \left[\varphi(\overline{x}_N, y^*)  \right] -
    \E \left[\varphi(x^*, \overline{y}_N) \right] \leq \max_{y\in \mathcal{Y}} \, \E \left[\varphi(\overline{x}_N, y)  \right] -
    \min_{x\in \mathcal{X}}\E \left[\varphi(x, \overline{y}_N) \right]  \\
     &=
    \max_{x\in \mathcal{X}, y\in \mathcal{Y}} \E
    \left[ \varphi_\mu\left(\overline{x}_N , y\right) - \varphi_\mu\left(x, \overline{y}_N\right) - \frac{\mu x_N^2}{2} + \frac{\mu y^2}{2} + \frac{\mu x^2}{2} - \frac{\mu y_N^2}{2} \right]  \\
     & \leq \max_{x\in \mathcal{X}, y\in \mathcal{Y}} \mathbb{E}  \left[ \varphi_\mu\left(\overline{x}_N , y\right) - \varphi_\mu\left(x, \overline{y}_N\right)  + \frac{\mu z^2}{2} \right]\\
     &\leq \max_{x\in \mathcal{X}, y\in \mathcal{Y}} \mathbb{E}  \left[ \varphi_\mu\left(\overline{x}_N , y\right) - \varphi_\mu\left(x, \overline{y}_N\right) \right] + \frac{\varepsilon}{2}\\
     &= \max_{y\in \mathcal{Y}} \E \left[\varphi_\mu(\overline{x}_N, y)  \right] -
    \min_{x\in \mathcal{X}}\E \left[\varphi_\mu(x, \overline{y}_N) \right] + \frac{\varepsilon}{2}\\
    \end{split}
\end{equation}

{\bf Step 2.} 
Now we apply Theorem \ref{theorem_highorder_smooth_strongly_convex} for $\varphi_{\mu}(z)$ until function error is not greater than $\frac{\varepsilon}{2}$:
\begin{equation}\label{ref-theorem-1}
    \max_{y\in \mathcal{Y}} \E \left[\varphi_\mu(\overline{x}_N, y)  \right] -
    \min_{x\in \mathcal{X}}\E \left[\varphi_\mu(x, \overline{y}_N) \right] \leq \dfrac{1}{\mu} \left(n^{2-\frac{1}{\beta}}\dfrac{A_1}{N^{\frac{\beta-1}{\beta}}}+A_2\dfrac{n(1+\ln{N})}{N} \right) \leq \dfrac{\varepsilon}{2}.
\end{equation}

Using that $\mu=\frac{\varepsilon}{R^2}$ the inequality \eqref{ref-theorem-1} is done if
\begin{equation}\label{th2-step-2}
    \max\left\{n^{2-\frac{1}{\beta}}\dfrac{A_1}{N^{\frac{\beta-1}{\beta}}}, A_2\dfrac{n(1+\ln{N})}{N} \right\} 
    \leq \dfrac{\mu \varepsilon}{2}
    = \dfrac{\varepsilon^2}{2R^2} .
\end{equation}

It is true that  $1 + \ln N \leq c' N^{\frac{\rho}{\rho + 1}}$ for some $c' > 0$. So the inequality \eqref{th2-step-2} holds if
\begin{equation}
    N \geq \max\left\{ \left(R\sqrt{2A_1}\right)^{\frac{2\beta}{\beta-1}}\dfrac{n^{2+\frac{1}{\beta-1}}}{\varepsilon^{2+\frac{2}{\beta-1}}},\left(R\sqrt{2c'A_2}\right)^{2(1+\rho)}\dfrac{n^{1+\rho}}{\varepsilon^{2(1+\rho)}}\right\}.
\end{equation}

The inequalities \eqref{connection} and \eqref{ref-theorem-1} yield \eqref{error_less_than_eps}.
\EndProof
\end{proof}

\section{Kernel examples}\label{appendix:kernels}

A weighted sum of Legendre polynoms is an example of such kernels:
\begin{equation}\label{kernels}
    K_{\beta}(r) := \sum\limits_{m=0}^{l(\beta)} p_{m}'(0)p_m(r),
\end{equation}
where $l(\beta)$ is maximal integer number strictly less than $\beta$ and $p_m(r) = \sqrt{2m+1} L_m(r)$, $L_m(u)$ is Legendre polynom. We have
\begin{equation*}
    \mathbb E \left[ p_m p_{m'}\right]=\delta(m-m').
\end{equation*}

As $\{p_m(r)\}_{m=0}^{j}$ is a basis for polynoms of degree less than or equal to $j$ we can represent $u^j := \sum\limits_{m=0}^{j} b_m p_m(r)$ for some integers $\{b_m\}_{m=0}^{j}$ (they depend on $j$). 

Let's calculate the expectation

\begin{equation*}
    \mathbb E \left[ r^j K_{\beta}(r)\right] = \sum\limits_{m=0}^{j} b_m p_m'(0) = (r^j)'|_{r=0} = \delta(j-1),
\end{equation*}
here $\delta(0) = 1$ and $\delta(x) = 1$ if $x \neq 0$. We proved that the presented $K_{\beta}(r)$ satisfies \eqref{kernel-properties}. We have the following kernels for different betas (see Figure \ref{fig:kernels}):

\begin{eqnarray*}
    &K_{\beta}(r) = 3r, \quad &\beta \in [2, 3],\\
    &K_{\beta}(r) = \dfrac{15r}{4} (5-7r^2), \quad &\beta \in (3, 5],\\
    &K_{\beta}(r) = \dfrac{105r}{64} (99r^4 - 126r^2 + 35), \quad &\beta \in (5, 7].
\end{eqnarray*}

\begin{figure}[h]
\centering
\includegraphics[width =  0.78\linewidth]{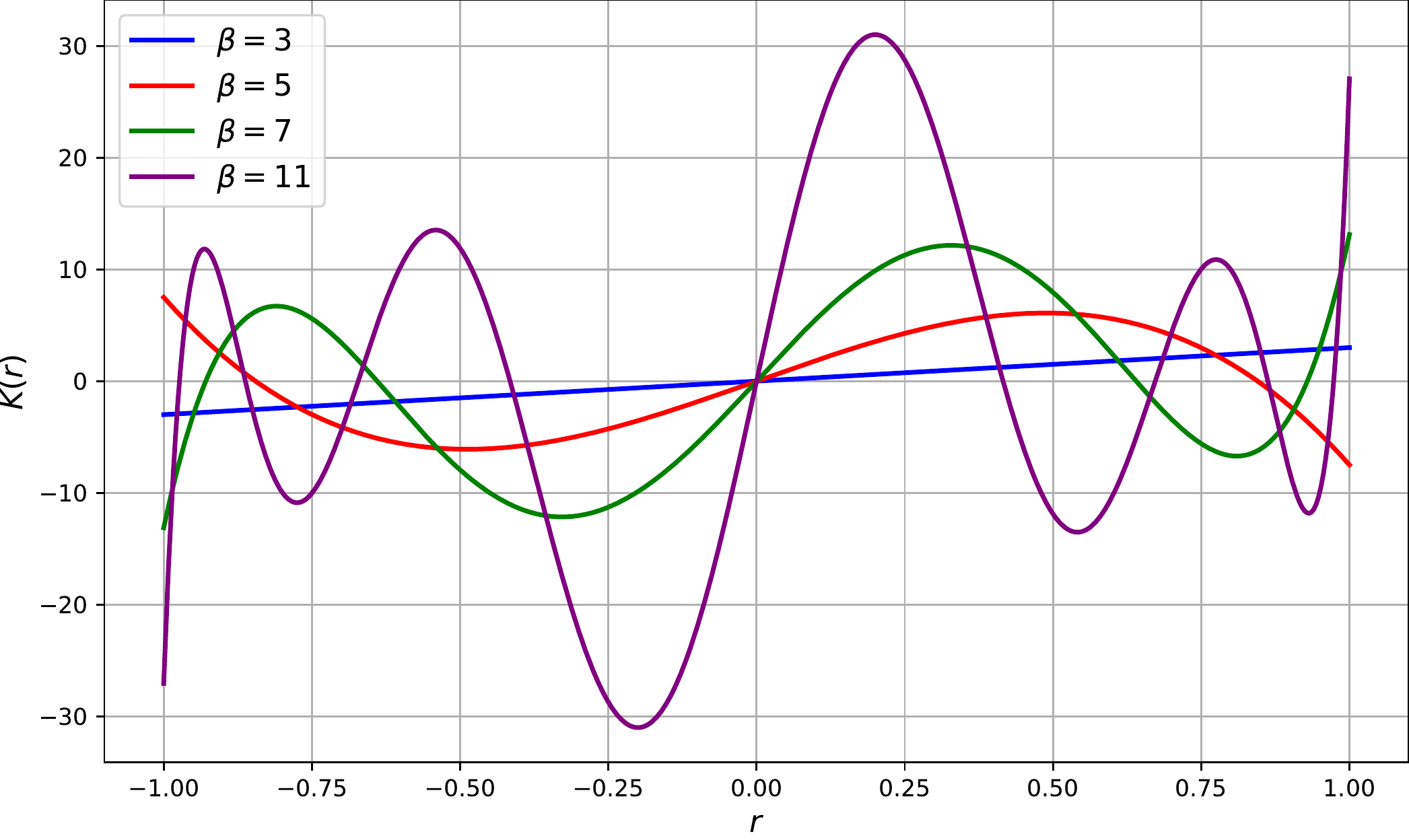}
\caption{Examples of kernels from \eqref{kernels}}
\label{fig:kernels}
\end{figure}

\end{document}